\documentclass[12pt]{article}
\usepackage{amsmath}
\usepackage{amsthm}
\usepackage{amsfonts}
\usepackage{amssymb}
\usepackage{graphicx}
\usepackage{tikz}
\usepackage{tikz-cd}
\usepackage{mathtools}
\usepackage{mathrsfs}
\usepackage{amsrefs}
\usepackage[none]{hyphenat}
\usepackage{hyperref}
\usepackage[capitalise, nameinlink]{cleveref-usedon}
\newtheorem{theorem}{Theorem}[section]
\newtheorem{corollary}[theorem]{Corollary}
\newtheorem{lemma}[theorem]{Lemma}
\newtheorem{prop}[theorem]{Proposition}
\newtheorem{conj}[theorem]{Conjecture}
\newtheorem{remark}[theorem]{Remark}
\newtheorem*{innerthm}{\innerthmname}
\newcommand{\innerthmname}{}
\newenvironment{statement}[1]
 {\renewcommand{\innerthmname}{#1}\innerthm}
 {\endinnerthm}
\theoremstyle{definition}
\newtheorem{definition}[theorem]{Definition}

\title{Properties of the Beilinson\\Height Pairing}
\author{Thomas Wisson}
\date{}

\begin{document}
\maketitle

\begin{abstract}
First constructed by Beilinson for curves defined over an algebraically closed field; R\"ossler and Szamuely generalized Beilinson's height pairing to the higher dimensional setting. In this paper we study this pairing and relate it to the intersection product by constructing a new height pairing. We also prove that it satisfies a projection formula matching the one satisfied by the intersection product.
\end{abstract}

\section{Introduction}
Let $X$ be a smooth projective variety of dimension $d$ over a function field $K$, i.e. $K=k(B)$ where $B$ is a smooth integral scheme of dimension $b$, of finite type over an algebraically closed field $k$.\\

We let $CH^{i}(X)$ denote the Chow groups of codimension $i$ cycles on $X$, and $CH^{i}_{\text{hom}}(X)_{\mathbb{Q}}$ be the kernel of the cycle map
\begin{align*}
CH^{i}(X)_{\mathbb{Q}}\rightarrow H^{2i}_{\text{\'et}}(X_{\overline{K}},\mathbb{Q}_{\ell}(i))
\end{align*}
where $\ell$ is a prime invertible in $K$, and $\overline{K}$ is the algebraic closure of $K$.\\

A height pairing on $X$ is a bilinear pairing
\begin{align*}
\langle\cdot,\cdot\rangle:CH^{p}_{\text{hom}}(X)_{\mathbb{Q}}\times CH^{q}_{\text{hom}}(X)_{\mathbb{Q}}\rightarrow \mathcal{Q}
\end{align*}
for each $p,q$ such that $p+q=d+1$, where $\mathcal{Q}$ is some coefficient group.
\subsection{The Geometric Height Pairing}
\label{sec:GeoHght}
There is a natural way to construct a height pairing on $X$, but this comes with a caveat. Take a regular projective model $\pi:\mathcal{X}\rightarrow B$ of $X$. We remark that such a model can always be found provided resolution of singularities holds, which is at least always true in characteristic zero.\\

Now, extend the cycles $a_{1}\in CH^{p}_{\text{hom}}(X)_{\mathbb{Q}}$, $a_{2}\in CH^{q}_{\text{hom}}(X)_{\mathbb{Q}}$ to cycles $\widetilde{a}_{1}\in CH^{p}(\mathcal{X})_{\mathbb{Q}}$, $\widetilde{a}_{2}\in CH^{q}(\mathcal{X})_{\mathbb{Q}}$ on $\mathcal{X}$. On the Chow groups of $\mathcal{X}$ the intersection product defines a bilinear pairing
\begin{align*}
\cdot\cap\cdot:CH^{p}(\mathcal{X})_{\mathbb{Q}}\times CH^{q}(\mathcal{X})_{\mathbb{Q}}\rightarrow CH^{d+1}(\mathcal{X})_{\mathbb{Q}}
\end{align*}
So we can define $\langle a_{1}, a_{2}\rangle_{\mathcal{X}}$ to be the intersection product $\widetilde{a}_{1}\cap\widetilde{a}_{2}$, or rather the image $\pi_{*}(\widetilde{a}_{1}\cap\widetilde{a}_{2})$ under the pushforward map $CH^{d+1}(\mathcal{X})_{\mathbb{Q}}\rightarrow CH^{1}(B)_{\mathbb{Q}}$.
\\

If $\mathcal{X}\rightarrow B$ is smooth this gives a well-defined height pairing by the following lemma, which is Proposition 6.1.~in \cite{Ros-Sam}. 
\begin{lemma}
\label{lem:SmthInt}
Suppose that $\pi:\mathcal{X}\rightarrow B$ is a smooth model. Then for any extension $\widetilde{a}_{1}\in CH^{p}(\mathcal{X})$ of $a_{1}\in CH^{p}_{\textnormal{hom}}(X)$, and any extension $\widetilde{a}_{2}\in CH^{q}(\mathcal{X})$ of $0\in CH^{q}(X)$, we have
\begin{align*}
\pi_{*}(\widetilde{a}_{1}\cap\widetilde{a}_{2})=0
\end{align*}
\end{lemma}
If the model is not smooth however, the pairing may depend on how we extend the $a_{i}$ to $\widetilde{a}_{i}$. This leads us on to the following conjecture.
\begin{conj}
\label{conj:BeilCyc}
For each $a_{1}\in CH^{p}_{\textnormal{hom}}(X)_{\mathbb{Q}}$ there exists a so-called Beilinson extension $\widetilde{a}_{1}\in CH^{p}(\mathcal{X})_{\mathbb{Q}}$ such that for any extension $\widetilde{a}_{2}\in CH^{q}(\mathcal{X})_{\mathbb{Q}}$ of $0\in CH^{q}(X)$ we have
\begin{align*}
\pi_{*}(\widetilde{a_{1}}\cap\widetilde{a}_{2})=0
\end{align*}
\end{conj}
\begin{remark}
\label{rem:BeilCycQ}
Note that this conjecture would be false if we did not tensor the Chow groups with $\mathbb{Q}$.
\end{remark}
Assuming the truth of this conjecture, we could promote the geometric height pairing to a genuine height pairing by defining $\langle a_{1}, a_{2}\rangle_{\mathcal{X}}$ to be $\pi_{*}(\widetilde{a}_{1}\cap\widetilde{a}_{2})$, where $\widetilde{a}_{1}$ is a Beilinson extension. For the meantime though, we must keep in mind the dependence on the choice of extension.\\

Work by Kahn in \cite{Kah} has shown that one can define a genuine height pairing on a certain subgroup $CH^{p}(X)^{(0)}\subseteq CH^{p}(X)$:
\begin{align*}
CH^{p}(X)^{(0)}_{\mathbb{Q}}\times CH^{q}(X)^{(0)}_{\mathbb{Q}}\rightarrow\text{Pic}(B)_{\mathbb{Q}}
\end{align*}
with values in $\text{Pic}(B)_{\mathbb{Q}}$, but without using Beilinson cycles. Moreover, it is equal to the Beilinson height pairing in a sense made precise in Proposition 2.11.~of \cite{Kah}.
\subsection{R\"ossler and Szamuely's Pairing}
\label{sec:BeilPair}
Let us now outline the construction of R\"ossler and Szamuely in \cite{Ros-Sam}.\\

Let $\mathcal{X}\xrightarrow{\pi} U$ be a smooth projective model of $X$ over an open affine subset $j:U\hookrightarrow B$. Then take the duality pairing on $U$:
\begin{align*}
\textbf{R}^{2p-1}\pi_{*}\mathbb{Q}_{\ell}(p)[b] \otimes^{\mathbf{L}} \textbf{R}^{2q-1}\pi_{*}\mathbb{Q}_{\ell}(q)[b]\rightarrow\mathbb{Q}_{\ell}(1)[2b]
\end{align*}
which under the Tensor-Hom adjunction corresponds to a mapping
\begin{align*}
\textbf{R}^{2p-1}\pi_{*}\mathbb{Q}_{\ell}(p)[b]\rightarrow D(\textbf{R}^{2q-1}\pi_{*}\mathbb{Q}_{\ell}(q)[b])
\end{align*}
where $D(-)=\mathbf{R}\textit{Hom}(-,\mathbb{Q}_{\ell}(b)[2b])$ is the duality functor. Now applying the intermediate extension functor $j_{!*}$, using the fact that $j_{!*}$ commutes with duality we get
\begin{align*}
j_{!*}\textbf{R}^{2p-1}\pi_{*}\mathbb{Q}_{\ell}(p)[b]\rightarrow D(j_{!*}\textbf{R}^{2q-1}\pi_{*}\mathbb{Q}_{\ell}(q)[b])
\end{align*}
which corresponds to a pairing
\begin{align*}
j_{!*}\textbf{R}^{2p-1}\pi_{*}\mathbb{Q}_{\ell}(p)[b] \otimes^{\mathbf{L}} j_{!*}\textbf{R}^{2q-1}\pi_{*}\mathbb{Q}_{\ell}(q)[b]\rightarrow\mathbb{Q}_{\ell}(1)[2b]
\end{align*}
in $D^{b}_{c}(B,\mathbb{Q}_{\ell})$ which we call $h_{U}$.\\

Finally, by taking cohomology this descends to a pairing
\begin{align*}
H^{1-b}_{\text{\'et}}(B,j_{!*}\textbf{R}^{2p-1}\pi_{*}\mathbb{Q}_{\ell}(p)[b]) \times H^{1-b}_{\text{\'et}}(B,j_{!*}\textbf{R}^{2q-1}\pi_{*}\mathbb{Q}_{\ell}(q)[b])\rightarrow H^{2}_{\text{\'et}}(B,\mathbb{Q}_{\ell}(1))
\end{align*}
which we shall also denote by $h_{U}(-,-)$.\\

In order then to define a height pairing on the Chow groups, they construct a class $\alpha^{p}_{U}\in H^{1-b}_{\text{\'et}}(B,j_{!*}\textbf{R}^{2p-1}\pi_{*}\mathbb{Q}_{\ell}(p)[b])$ from an element $\alpha^{p}\in CH^{p}_{\text{hom}}(X)_{\mathbb{Q}}$ (and likewise for $q$), and define
\begin{align*}
\langle\alpha^{p},\alpha^{q}\rangle_{U}:=h_{U}(\alpha^{p}_{U},\alpha^{q}_{U})
\end{align*}
In more detail, $\alpha^{p}_{U}$ is defined in the following way. One takes a representative $z$ of $\alpha^{p}$ and extends it to a cycle $z_{U}$ on the model $\mathcal{X}$ (we can always shrink $U$ if necessary). Since $z$ is homologically trivial, by the smooth proper base change theorem $\text{cl}(z_{U})\in H^{2p}_{\text{\'et}}(\mathcal{X},\mathbb{Q}_{\ell}(p))$ lies in the kernel of the map
\begin{align*}
H^{2p}_{\text{\'et}}(\mathcal{X},\mathbb{Q}_{\ell}(p))\rightarrow H^{0}_{\text{\'et}}(U,\mathbf{R}^{2p}\pi_{*}\mathbb{Q}_{\ell}(p))
\end{align*}
and therefore produces an element $\alpha^{p}_{U}\in H^{1}_{\text{\'et}}(U,\mathbf{R}^{2p-1}\pi_{*}\mathbb{Q}_{\ell}(p))$, which they in turn prove always lies inside the subspace
\begin{align*}
H^{1-b}_{\text{\'et}}(B,j_{!*}\textbf{R}^{2p-1}\pi_{*}\mathbb{Q}_{\ell}(p)[b])\subset H^{1}_{\text{\'et}}(U,\mathbf{R}^{2p-1}\pi_{*}\mathbb{Q}_{\ell}(p))
\end{align*}
Finally, let us briefly discuss how one could use this approach to build a conjectural motivic height pairing, i.e. a pairing
\begin{align*}
\langle\cdot,\cdot\rangle:CH^{p}_{\text{hom}}(X)_{\mathbb{Q}}\times CH^{q}_{\text{hom}}(X)_{\mathbb{Q}}\rightarrow \text{Pic}(B)_{\mathbb{Q}}
\end{align*}
This construction hinges on the existence of a category $M(B,\mathbb{Q})$ of motivic $\mathbb{Q}$-sheaves analogous to the category of $\mathbb{Q}_{\ell}$-sheaves, as described in Chapter 5.10.~of \cite{Bei}. One also expects a perverse t-structure on the derived category of $M(B,\mathbb{Q})$, which would allow us to define a corresponding intermediate extension functor.\\

One could then, by the same argument, define a pairing
\begin{align*}
H^{1-b}(B,j_{!*}\textbf{R}^{2p-1}\pi_{*}\mathbb{Q}(p)[b]) \times H^{1-b}(B,j_{!*}\textbf{R}^{2q-1}\pi_{*}\mathbb{Q}(q)[b])\rightarrow H^{2}(B,\mathbb{Q}(1))=\text{Pic}(B)_{\mathbb{Q}}
\end{align*}
where of course $H^{i}(B,-)=\text{Ext}^{i}_{M(B,\mathbb{Q})}(\mathbb{Q},-)$. This pairing then induces the one on the Chow groups.\\

A candidate for the category $\text{Perv}(B,\mathbb{Q})$ in the case when $k$ is of characteristic zero, with an embedding $k\hookrightarrow\mathbb{C}$, is given by Ivorra and Morel in \cite{Ivo-Mor}. See Definition 2.1.~in their paper for a definition of the category of perverse motives $\mathcal{M}(B)$. They also prove in Theorem 5.1.~that on the associated bounded derived category $D^{b}(\mathcal{M}(B))$ exists a pullback, pushforward, exceptional pullback and exceptional pushforward functor.
\subsection{Summary of the Main Results}
\label{sec:Summ}
The main content of \cref{sec:CompHght} is to construct a third height pairing $\langle -,-\rangle_{B}$ which can act as an intermediary to compare R\"ossler and Szamuely's pairing to the geometric one. Indeed, we have the following two results
\begin{statement}{\cref{th:Coh=Geo}}
For any $\alpha^{p}\in CH^{p}_{\textnormal{hom}}(X)_{\mathbb{Q}}$, $\alpha^{q}\in CH^{q}_{\textnormal{hom}}(X)_{\mathbb{Q}}$ we have
\begin{align*}
\langle\alpha^{p},\alpha^{q}\rangle_{B}=\textnormal{cl}_{B}\langle\alpha^{p},\alpha^{q}\rangle_{\mathcal{X}}
\end{align*}
\end{statement}
\begin{statement}{\cref{th:CompCoh}}
For any $\alpha^{p}\in CH^{p}_{\textnormal{hom}}(X)_{\mathbb{Q}}$, $\alpha^{q}\in CH^{q}_{\textnormal{hom}}(X)_{\mathbb{Q}}$ we have
\begin{align*}
\langle\alpha^{p},\alpha^{q}\rangle_{B}=\langle\alpha^{p},\alpha^{q}\rangle_{U}+h_{B}(\widehat{\alpha^{p}_{B}},\widehat{\alpha^{q}_{B}})
\end{align*}
Moreover, if we can find a representative $z^{p}_{B}$ such that the image of $\textnormal{cl}(z^{p}_{B})$ inside $H^{0}_{\textnormal{\'et}}(B,\textnormal{\textbf{R}}^{2p}\pi_{*}\mathbb{Q}_{\ell}(p)))$ is zero, then for the corresponding $\alpha^{p}_{B}$ we have $\widehat{\alpha^{p}_{B}}=0$, and consequently
\begin{align*}
\langle\alpha^{p},\alpha^{q}\rangle_{B}=\langle\alpha^{p},\alpha^{q}\rangle_{U}
\end{align*}
\end{statement}
Here $h_{B}$ is the cohomological pairing from which $\langle -,-\rangle_{B}$ is defined (in the same way that $\langle -,-\rangle_{U}$ is defined from $h_{U}$). The meaning of $\widehat{\alpha^{p}_{B}}$ is explained in \cref{sec:CohPair}. Putting these two theorems together we have
\begin{statement}{\cref{cor:Geo=Perv}}
For any $\alpha^{p}\in CH^{p}_{\textnormal{hom}}(X)_{\mathbb{Q}}$, $\alpha^{q}\in CH^{q}_{\textnormal{hom}}(X)_{\mathbb{Q}}$ we have
\begin{align*}
\textnormal{cl}_{B}\langle\alpha^{p},\alpha^{q}\rangle_{\mathcal{X}}=\langle\alpha^{p},\alpha^{q}\rangle_{U}+h_{B}(\widehat{\alpha^{p}_{B}},\widehat{\alpha^{q}_{B}})
\end{align*}
\end{statement}
\cref{sec:ProjForm} meanwhile is devoted to proving a projection formula for R\"ossler and Szamuely's pairing, which we also denote as $\langle -,-\rangle_{X}$.
\begin{statement}{\cref{th:ProjForm}}
For any $\alpha^{p}\in CH_{\textnormal{hom}}^{p}(X')_{\mathbb{Q}}$ and $\beta^{q}\in CH_{\textnormal{hom}}^{q}(X)_{\mathbb{Q}}$ the projection formula holds
\begin{align*}
\langle g_{*}\alpha^{p},\beta^{q}\rangle_{X}=f_{*}\langle \alpha^{p},g^{*}\beta^{q}\rangle_{X'}
\end{align*}
\end{statement}
where here $f:B'\rightarrow B$ is a dominant proper morphism of relative dimension $0$ which is generically finite (with $K'\rightarrow K$ the induced map on function fields), and $g:X'\rightarrow X$ is a finite map which fits into a commutative diagram
\begin{center}
\begin{tikzcd}
  X' \arrow[r, "g"] \arrow[d]
  & X \arrow[d] \\
  K' \arrow[r]
  & K
\end{tikzcd}
\end{center}
\section{Comparing Height Pairings}
\label{sec:CompHght}
In order to compare the pairing $\langle -,-\rangle_{U}$ (which we shall refer to simply as Beilinson's pairing) to the geometric one, we define a new height pairing as follows.\\

For this section, let $\mathcal{X}\xrightarrow{\pi} B$ be a regular projective model (which may not be smooth), and let $U$ be an open affine subset of $B$ such that $\mathcal{X}_{U}\xrightarrow{\pi_{U}} U$ is a smooth model of $X$. Also, let $j:U\hookrightarrow B$ the open embedding, and $i:Z=B\setminus U\hookrightarrow B$ the closed immersion.\\

We shall make use of the decomposition theorem for $\pi$ (one can find a proof in \cite{Bei-Del}, or see Theorem 1.3.1.~in \cite{Cat} for an overview), which gives us
\begin{align}
\label{eq:DecompThm}
\mathbf{R}\pi_{*}\mathbb{Q}_{\ell}(p)\simeq\bigoplus_{k\geq 0} \big(j_{!*}(\mathbf{R}^{k}\pi_{U*}\mathbb{Q}_{\ell}(p)[b])\oplus C_{k}\big)[-k-b]
\end{align}
where each $C_{k}=i_{*}\widehat{C}_{k}$ is a perverse sheaf supported on $Z$. Indeed, one can check this by restricting to $U$ and noting that the decomposition theorem for $\pi_{U}$ tells us
\begin{align*}
\mathbf{R}\pi_{U*}\mathbb{Q}_{\ell}(p)\simeq\bigoplus_{k\geq 0} \mathbf{R}^{k}\pi_{U*}\mathbb{Q}_{\ell}(p)[-k]
\end{align*}
\subsection{A New Cohomological Pairing}
First of all, consider the perverse spectral sequence
\begin{align*}
H^{s}_{\text{\'et}}(B,{}^{\mathfrak{p}}\mathcal{H}^{t}(\textbf{R}\pi_{*}\mathbb{Q}_{\ell}(p)))\Rightarrow H^{s+t}_{\text{\'et}}(\mathcal{X},\mathbb{Q}_{\ell}(p))
\end{align*}
Given $\alpha^{p}\in CH^{p}_{\text{hom}}(X)_{\mathbb{Q}}$, extending a representative $z^{p}$ of $\alpha^{p}$ to a cycle $z^{p}_{B}$ on the model $\mathcal{X}\rightarrow B$, we get a class $\text{cl}(z^{p}_{B})\in H^{2p}_{\text{\'et}}(\mathcal{X},\mathbb{Q}_{\ell}(p))$.
\begin{lemma}
\label{lem:Vnsh}
We have that
\begin{align*}
H^{c}_{\textnormal{\'et}}(B,{}^{\mathfrak{p}}\mathcal{H}^{2p-c}(\textnormal{\textbf{R}}\pi_{*}\mathbb{Q}_{\ell}(p)))=0
\end{align*}
for all $c<-b$.
\end{lemma}
\begin{proof}
Applying Lemma III.5.13.~in \cite{Kie} to $\mathcal{G}={}^{\mathfrak{p}}\mathcal{H}^{2p-c}(\mathbf{R}\pi_{*}\mathbb{Q}_{\ell}(p))$ we get that
\begin{align*}
\mathcal{H}^{c}(\mathcal{G})=0
\end{align*}
for all $c<-b$ since $\text{dim}(B)\leq b$. So using the hypercohomolgy spectral sequence
\begin{align*}
E^{s,t}_{2}=H^{s}_{\text{\'et}}(B,\mathcal{H}^{t}(\mathcal{G}))\Rightarrow H^{s+t}_{\text{\'et}}(B,\mathcal{G})
\end{align*}
since $H^{s}_{\text{\'et}}(B,\mathcal{H}^{t}(\mathcal{G}))=0$ for $s<0$ or $t<-b$, the lemma follows.

\end{proof}
Hence the first non-zero term with grading $2p$ is $H^{-b}_{\text{\'et}}(B,{}^{\mathfrak{p}}\mathcal{H}^{2p+b}(\textbf{R}\pi_{*}\mathbb{Q}_{\ell}(p)))$.
\begin{lemma}
\label{lem:ImCyc}
The image of $\textnormal{cl}(z^{p}_{B})$ in $H^{-b}_{\textnormal{\'et}}(B,{}^{\mathfrak{p}}\mathcal{H}^{2p+b}(\textnormal{\textbf{R}}\pi_{*}\mathbb{Q}_{\ell}(p)))$ is zero.
\end{lemma}
\begin{proof}
Using the decomposition theorem \eqref{eq:DecompThm} we have that
\begin{align*}
{}^{\mathfrak{p}}\mathcal{H}^{2p+b}(\textbf{R}\pi_{*}\mathbb{Q}_{\ell}(p))&\simeq\bigoplus_{k\geq 0}{}^{\mathfrak{p}}\mathcal{H}^{2p-k}(j_{!*}(\textbf{R}^{k}\pi_{U*}\mathbb{Q}_{\ell}(p)[b])\oplus C_{k})\\
&=j_{!*}\textbf{R}^{2p}\pi_{U*}\mathbb{Q}_{\ell}(p)[b]\oplus C_{2p}
\end{align*}
where the second equality uses the fact that the $j_{!*}(\textbf{R}^{k}\pi_{U*}\mathbb{Q}_{\ell}(p)[b])$ and $C_{k}$ are perverse sheaves. Note that since $C_{2p}$ is supported on $Z$, by the same argument as \cref{lem:Vnsh}, using that $\text{dim}(Z)\leq b-1$, we have
\begin{align*}
H^{-b}_{\text{\'et}}(B,C_{2p})=0
\end{align*}
which gives us that
\begin{align*}
H^{-b}_{\text{\'et}}(B,{}^{\mathfrak{p}}\mathcal{H}^{2p+b}(\textbf{R}\pi_{*}\mathbb{Q}_{\ell}(p)))=H^{-b}_{\text{\'et}}(B,j_{!*}\textbf{R}^{2p}\pi_{U*}\mathbb{Q}_{\ell}(p)[b])
\end{align*}
Using the distinguished triangle in Lemma III.5.1.~of \cite{Kie} for $\mathcal{F}=\textbf{R}^{2p}\pi_{U*}\mathbb{Q}_{\ell}(p)[b]$, we get an exact sequence of cohomology groups
\begin{align*}
H^{-1-b}_{\text{\'et}}(B,i_{*}{}^{\mathfrak{p}}\tau_{\geq 0}i^{*}\mathbf{R}j_{*}\mathcal{F})\rightarrow H^{-b}_{\text{\'et}}(B,j_{!*}\mathcal{F})\rightarrow H^{-b}_{\text{\'et}}(B,\mathbf{R}j_{*}\mathcal{F})
\end{align*}
where because ${}^{\mathfrak{p}}\tau_{\geq 0}i^{*}\mathbf{R}j_{*}\mathcal{F}$ is supported on $Z$, and $i_{*}$ is exact, we must have that the left hand group vanishes $H^{-1-b}_{\text{\'et}}(B,i_{*}{}^{\mathfrak{p}}\tau_{\geq 0}i^{*}\mathbf{R}j_{*}\mathcal{F})=0$.\\

Hence $H^{-b}_{\text{\'et}}(B,{}^{\mathfrak{p}}\mathcal{H}^{2p+b}(\textbf{R}\pi_{*}\mathbb{Q}_{\ell}(p)))$ injects into $H^{0}_{\text{\'et}}(U,\textbf{R}^{2p}\pi_{U*}\mathbb{Q}_{\ell}(p))$. So now, using the map of spectral sequences
\begin{equation}
\label{eq:SpecSeq}
\begin{tikzcd}
  H^{i-b}_{\text{\'et}}(B,{}^{\mathfrak{p}}\mathcal{H}^{2p-i+b}(\textbf{R}\pi_{*}\mathbb{Q}_{\ell}(p))) \arrow[Rightarrow, shorten=5]{r} \arrow[d]
  & H^{2p}_{\text{\'et}}(\mathcal{X},\mathbb{Q}_{\ell}(p)) \arrow[d] \\
  H^{i}_{\text{\'et}}(U,\textbf{R}^{2p-i}\pi_{*}\mathbb{Q}_{\ell}(p)) \arrow[Rightarrow, shorten=5]{r}
  & H^{2p}_{\text{\'et}}(\mathcal{X}_{U},\mathbb{Q}_{\ell}(p))
\end{tikzcd}
\end{equation}
we get a commutative diagram
\begin{center}
\begin{tikzcd}
  H^{2p}_{\text{\'et}}(\mathcal{X},\mathbb{Q}_{\ell}(p)) \arrow[d] \arrow[r] 
  & H^{2p}_{\text{\'et}}(\mathcal{X}_{U},\mathbb{Q}_{\ell}(p)) \arrow[d] \\ 
   H^{-b}_{\text{\'et}}(B,{}^{\mathfrak{p}}\mathcal{H}^{2p+b}(\textbf{R}\pi_{*}\mathbb{Q}_{\ell}(p))) \arrow[hookrightarrow]{r}
  & H^{0}_{\text{\'et}}(U,\mathbf{R}^{2p}\pi_{U*}\mathbb{Q}_{\ell}(p))
\end{tikzcd}
\end{center}
where because the image of $\text{cl}(z^{p}_{U})$ in $H^{0}_{\text{\'et}}(U,\mathbf{R}^{2p}\pi_{U*}\mathbb{Q}_{\ell}(p))$ is zero (this is implied by the smooth proper base change theorem since $z^{p}$ is homologically trivial), the image of $\textnormal{cl}(z^{p}_{B})$ inside $H^{-b}_{\text{\'et}}(B,{}^{\mathfrak{p}}\mathcal{H}^{2p+b}(\textbf{R}\pi_{*}\mathbb{Q}_{\ell}(p)))$ is also zero, and the lemma follows.

\end{proof}
Let $P_{p}=\text{Ker}(H^{2p}_{\text{\'et}}(\mathcal{X},\mathbb{Q}_{\ell}(p))\rightarrow H^{-b}_{\text{\'et}}(B,{}^{\mathfrak{p}}\mathcal{H}^{2p+b}(\textbf{R}\pi_{*}\mathbb{Q}_{\ell}(p)))$ and likewise for $P_{q}$. By the lemma then we have that $\text{cl}(z^{p}_{B})\in P_{p}$, $\text{cl}(z^{q}_{B})\in P_{q}$.
\begin{definition}
\label{def:PervCyc}
We define the class
\begin{align*}
\alpha^{p}_{B}\in H^{1-b}_{\text{\'et}}(B,{}^{\mathfrak{p}}\mathcal{H}^{2p-1+b}(\textbf{R}\pi_{*}\mathbb{Q}_{\ell}(p)))
\end{align*}
to be the image of $\text{cl}(z^{p}_{B})\in P_{p}$ under the map
\begin{align*}
P_{p}\rightarrow H^{1-b}_{\text{\'et}}(B,{}^{\mathfrak{p}}\mathcal{H}^{2p-1+b}(\textbf{R}\pi_{*}\mathbb{Q}_{\ell}(p)))
\end{align*}
and similarly for $q$.
\end{definition}

Now, starting with the tensor product pairing
\begin{align*}
\mathbb{Q}_{\ell}(p)\otimes\mathbb{Q}_{\ell}(q)\rightarrow\mathbb{Q}_{\ell}(d+1)
\end{align*}
this induces the derived pairing
\begin{align*}
\mathbf{R}\pi_{*}\mathbb{Q}_{\ell}(p)\otimes^{\textbf{L}}\mathbf{R}\pi_{*}\mathbb{Q}_{\ell}(q)\rightarrow\mathbf{R}\pi_{*}\mathbb{Q}_{\ell}(d+1)
\end{align*}
We have the trace map $\mathbf{R}\pi_{*}\mathbb{Q}_{\ell}(d+b)[2(d+b)]=\mathbf{R}\pi_{!}\pi^{!}\mathbb{Q}_{\ell}(b)[2b]\rightarrow\mathbb{Q}_{\ell}(b)[2b]$, where we have used that $\mathcal{X}$ is regular to deduce
\begin{align*}
\pi^{!}\mathbb{Q}_{\ell}(b)[2b]&=\pi^{!}K_{B}\\
&=\pi^{!}g^{!}K_{k}\\
&=(g\circ\pi)^{!}K_{k}=K_{\mathcal{X}}=\mathbb{Q}_{\ell}(d+b)[2(d+b)]
\end{align*}
Here $g:B\rightarrow k$ is the structure morphism and $K_{Y}$ denotes the dualizing complex on a scheme $Y$ over $k$. Indeed, $\mathcal{X}$ regular means that $g\circ\pi$ is smooth, yielding the final equality.\\

So we get a map
\begin{align*}
\mathbf{R}\pi_{*}\mathbb{Q}_{\ell}(d+1)\rightarrow\mathbb{Q}_{\ell}(1)[-2d]
\end{align*}
which if we compose with the previous pairing gives us
\begin{align}
\label{eq:DerPair}
\mathbf{R}\pi_{*}\mathbb{Q}_{\ell}(p)[2p-1+b]\otimes^{\textbf{L}}\mathbf{R}\pi_{*}\mathbb{Q}_{\ell}(q-1+b)[2q-1+b]\rightarrow\mathbb{Q}_{\ell}(b)[2b]
\end{align}
Since $\mathbb{Q}_{\ell}(b)[2b]$ is the dualizing sheaf on $B$, this corresponds to a map
\begin{align}
\label{eq:DerMap}
\mathbf{R}\pi_{*}\mathbb{Q}_{\ell}(p)[2p-1+b]\rightarrow D(\mathbf{R}\pi_{*}\mathbb{Q}_{\ell}(q-1+b)[2q-1+b])
\end{align}
Noting that ${}^{\mathfrak{p}}\tau_{\leq 0}D=D {}^{\mathfrak{p}}\tau_{\geq 0}$ and ${}^{\mathfrak{p}}\tau_{\geq 0}D=D {}^{\mathfrak{p}}\tau_{\leq 0}$ because of the self dual nature of the perverse t-structure, we see that ${}^{\mathfrak{p}}\mathcal{H}^{0}D=D {}^{\mathfrak{p}}\mathcal{H}^{0}$ and thus by taking perverse cohomology we have a map
\begin{align*}
{}^{\mathfrak{p}}\mathcal{H}^{0}(\mathbf{R}\pi_{*}\mathbb{Q}_{\ell}(p)[2p-1+b])\rightarrow D({}^{\mathfrak{p}}\mathcal{H}^{0}(\mathbf{R}\pi_{*}\mathbb{Q}_{\ell}(q-1+b)[2q-1+b]))
\end{align*}
which then corresponds to a pairing $h_{B}$
\begin{align*}
{}^{\mathfrak{p}}\mathcal{H}^{2p-1+b}(\mathbf{R}\pi_{*}\mathbb{Q}_{\ell}(p))\otimes^{\mathbf{L}}{}^{\mathfrak{p}}\mathcal{H}^{2q-1+b}(\mathbf{R}\pi_{*}\mathbb{Q}_{\ell}(q))\rightarrow\mathbb{Q}_{\ell}(1)[2b]
\end{align*}
Finally, taking cohomology gives us the cohomological pairing $h_{B}(-,-)$
\begin{align*}
H^{1-b}_{\text{\'et}}(B,{}^{\mathfrak{p}}\mathcal{H}^{2p-1+b}(\textbf{R}\pi_{*}\mathbb{Q}_{\ell}(p)))\times H^{1-b}_{\text{\'et}}(B,{}^{\mathfrak{p}}\mathcal{H}^{2q-1+b}(\textbf{R}\pi_{*}\mathbb{Q}_{\ell}(q)))\rightarrow H^{2}_{\text{\'et}}(B,\mathbb{Q}_{\ell}(1))
\end{align*}

We would like then to define the height pairing $\langle -,-\rangle_{B}$ by setting
\begin{align*}
\langle \alpha^{p},\alpha^{q}\rangle_{B}:=h_{B}(\alpha^{p}_{B},\alpha^{q}_{B})
\end{align*}
however one must be aware that this may depend on the choice of $z^{p}_{B}$ extending $z^{p}$, and likewise for $z^{q}_{B}$.\\

Finally, let us briefly describe how to proceed should one not have a regular projective model available, as could theoretically happen in the positive characteristic case. We can at least still take a projective model $\mathcal{X}\rightarrow B$ of $X$.\\

With the same notation as before, we consider $\text{cl}(z^{p}_{U})\in H^{2p}_{\text{\'et}}(\mathcal{X}_{U},\mathbb{Q}_{\ell}(p))$, where $z^{p}_{U}$ is some extension of $z^{p}$ on $\mathcal{X}_{U}$. Now, using de Jong's theorem on alterations \cite{dJo} to find a proper and generically \'etale morphism $\mathcal{W}\rightarrow\mathcal{X}$ over $k$ such that $\mathcal{W}$ is regular, we can prove the following.
\begin{lemma}
\label{lem:ResCyc}
The class $\textnormal{cl}(z^{p}_{U})$ lies in the image of the restriction map
\begin{align*}
H^{2p-b-d}_{\textnormal{\'et}}(\mathcal{X},\textnormal{IC}_{\mathcal{X}}(p))\rightarrow H^{2p}_{\textnormal{\'et}}(\mathcal{X}_{U},\mathbb{Q}_{\ell}(p))
\end{align*}
\end{lemma}
\begin{proof}
See Proposition 5.1.~in \cite{Ros-Sam}.

\end{proof}
Here $\text{IC}_{\mathcal{X}}=j_{\mathcal{X}!*}(\mathbb{Q}_{\ell}[b+d])$ is the intersection complex on $\mathcal{X}$, where $j_{\mathcal{X}}:\mathcal{X}_{U}\hookrightarrow\mathcal{X}$ is the inclusion. So using the lemma, we get a class in $H^{2p-b-d}_{\textnormal{\'et}}(\mathcal{X},\textnormal{IC}_{\mathcal{X}}(p))$, and by the same argument as before this gives rise to a class
\begin{align*}
\alpha^{p}_{B}\in H^{1-b}_{\text{\'et}}(B,{}^{\mathfrak{p}}\mathcal{H}^{2p-1-d}(\textbf{R}\pi_{*}\text{IC}_{\mathcal{X}}(p)))
\end{align*}
Next, noting that $D(\text{IC}_{\mathcal{X}})=\text{IC}_{\mathcal{X}}(b+d)$, we have the derived pairing (where $K_{\mathcal{X}}$ is the dualizing complex on $\mathcal{X}$)
\begin{align*}
\text{IC}_{\mathcal{X}}(p)\otimes^{\mathbf{L}}\text{IC}_{\mathcal{X}}(q)\rightarrow K_{\mathcal{X}}(1-b)
\end{align*}
which we can use to define a cohomological pairing $h_{B}$ in the same way as before
\begin{align*}
H^{1-b}_{\text{\'et}}(B,{}^{\mathfrak{p}}\mathcal{H}^{2p-1-d}(\textbf{R}\pi_{*}\text{IC}_{\mathcal{X}}(p)))\times H^{1-b}_{\text{\'et}}(B,{}^{\mathfrak{p}}\mathcal{H}^{2q-1-d}(\textbf{R}\pi_{*}\text{IC}_{\mathcal{X}}(q)))\rightarrow H^{2}_{\text{\'et}}(B,\mathbb{Q}_{\ell}(1))
\end{align*}
and then we can finally set $\langle \alpha^{p},\alpha^{q}\rangle_{B}:=h_{B}(\alpha^{p}_{B},\alpha^{q}_{B})$ just as before. All that follows applies equally to this pairing, but one must keep in mind that the intersection product is not well defined on $\mathcal{X}$ since it is not regular, so only the `cohomological realization' of the geometric pairing (that we define in the next section) is well defined.
\subsection{Cohomological vs Geometric}
The geometric pairing has a cohomological realization $h_{\mathcal{X}}$ which fits into the commutative diagram
\begin{center}
\begin{tikzcd}
  CH^{p}(\mathcal{X})\times CH^{q}(\mathcal{X}) \arrow[d] \arrow[r, "\cap"] 
  & CH^{1}(B) \arrow[d] \\
  H^{2p}_{\text{\'et}}(\mathcal{X},\mathbb{Q}_{\ell}(p))\times H^{2q}_{\text{\'et}}(\mathcal{X},\mathbb{Q}_{\ell}(q)) \arrow[r, "h_{\mathcal{X}}"']
  & H^{2}_{\text{\'et}}(B,\mathbb{Q}_{\ell}(1))
\end{tikzcd}
\end{center}
It is constructed from the derived pairing \eqref{eq:DerPair} by taking cohomology directly, and so one might hope that we can relate it to the cohomological pairing in the previous section.
\begin{theorem}
\label{th:Coh=Geo}
The cohomological pairing is equal to the geometric pairing
\begin{align*}
h_{B}(\alpha^{p}_{B},\alpha^{q}_{B})=h_{\mathcal{X}}(\textnormal{cl}(z^{p}_{B}),\textnormal{cl}(z^{q}_{B}))
\end{align*}where $\alpha^{p}_{B}$ is the image of $\textnormal{cl}(z^{p}_{B})$ under $P_{p}\rightarrow H^{1-b}_{\textnormal{\'et}}(B,{}^{\mathfrak{p}}\mathcal{H}^{2p-1+b}(\textnormal{\textbf{R}}\pi_{*}\mathbb{Q}_{\ell}(p)))$, and likewise for $q$.
\end{theorem}
\begin{prop}
\label{prop:Coh-Geom}
The following diagram commutes
\begin{center}
\begin{tikzcd}
  P_{p}\times P_{q} \arrow[dd] \arrow[dr, "h_{\mathcal{X}}"] 
  & ~ \\
  ~
  & H^{2}_{\textnormal{\'et}}(B,\mathbb{Q}_{\ell}(1)) \\  
  H^{1-b}_{\textnormal{\'et}}(B,{}^{\mathfrak{p}}\mathcal{H}^{2p-1+b}(\textnormal{\textbf{R}}\pi_{*}\mathbb{Q}_{\ell}(p)))\times H^{1-b}_{\textnormal{\'et}}(B,{}^{\mathfrak{p}}\mathcal{H}^{2q-1+b}(\textnormal{\textbf{R}}\pi_{*}\mathbb{Q}_{\ell}(q))) \arrow[ur, "h_{B}"']
  & ~
\end{tikzcd}
\end{center}
\end{prop}
\begin{proof}
First note that
\begin{align*}
H^{2p}_{\text{\'et}}(\mathcal{X},\mathbb{Q}_{\ell}(p))\simeq H^{2p}_{\text{\'et}}(B,\mathbf{R}\pi_{*}\mathbb{Q}_{\ell}(p))\simeq H^{2p}_{\text{\'et}}(B,{}^{\mathfrak{p}}\tau_{\leq 2p+b}\mathbf{R}\pi_{*}\mathbb{Q}_{\ell}(p))
\end{align*}
where the second isomorphism is from the long exact sequence associated to the exact triangle
\begin{align*}
{}^{\mathfrak{p}}\tau_{\leq 2p+b}\mathbf{R}\pi_{*}\mathbb{Q}_{\ell}(p)\rightarrow \mathbf{R}\pi_{*}\mathbb{Q}_{\ell}(p)\rightarrow {}^{\mathfrak{p}}\tau_{\geq 2p+1+b}\mathbf{R}\pi_{*}\mathbb{Q}_{\ell}(p)\rightarrow {}^{\mathfrak{p}}\tau_{\leq 2p+b}\mathbf{R}\pi_{*}\mathbb{Q}_{\ell}(p)[1]
\end{align*}
Indeed, the following groups vanish
\begin{lemma}
\label{lem:BndCoh}
We have
\begin{align*}
H^{c}_{\textnormal{\'et}}(B,{}^{\mathfrak{p}}\tau_{\geq 2p+1+b}\mathbf{R}\pi_{*}\mathbb{Q}_{\ell}(p))=0
\end{align*}
for all $c<2p+1$.
\end{lemma}
\begin{proof}
We use the same argument as in \cref{lem:Vnsh}. Let $\mathcal{G}={}^{\mathfrak{p}}\tau_{\geq 2p+1+b}\mathbf{R}\pi_{*}\mathbb{Q}_{\ell}(p)$, then by applying Lemma 5.1.13.~of \cite{Kie} to $\mathcal{G}[2p+1+b]$ we have that
\begin{align*}
\mathcal{H}^{c}(\mathcal{G})=0
\end{align*}
for all $c<2p+1$. Using the spectral sequence
\begin{align*}
H^{s}_{\text{\'et}}(B,\mathcal{H}^{t}(\mathcal{G}))\Rightarrow H^{s+t}_{\text{\'et}}(B,\mathcal{G})
\end{align*}
since $H^{s}_{\text{\'et}}(B,\mathcal{H}^{t}(\mathcal{G}))=0$ for $s<0$ or $t<2p+1$, the lemma now follows.

\end{proof}
Next, using the exact triangle 
\begin{align*}
{}^{\mathfrak{p}}\tau_{\leq 2p-1+b}\mathbf{R}\pi_{*}\mathbb{Q}_{\ell}(p)\rightarrow {}^{\mathfrak{p}}\tau_{\leq 2p+b}\mathbf{R}\pi_{*}\mathbb{Q}_{\ell}(p)\rightarrow {}^{\mathfrak{p}}\mathcal{H}^{2p+b}(\textbf{R}\pi_{*}\mathbb{Q}_{\ell}(p))[-2p-b]\xrightarrow{[1]}
\end{align*}
and the long exact cohomology sequence associated to it, we have that the kernel of 
\begin{align*}
H^{2p}_{\text{\'et}}(\mathcal{X},\mathbb{Q}_{\ell}(p))\simeq H^{2p}_{\text{\'et}}(B,{}^{\mathfrak{p}}\tau_{\leq 2p+b}\mathbf{R}\pi_{*}\mathbb{Q}_{\ell}(p))\rightarrow H^{-b}_{\text{\'et}}(B,{}^{\mathfrak{p}}\mathcal{H}^{2p+b}(\textbf{R}\pi_{*}\mathbb{Q}_{\ell}(p)))
\end{align*}
is equal to $H^{2p}_{\text{\'et}}(B,{}^{\mathfrak{p}}\tau_{\leq 2p-1+b}\mathbf{R}\pi_{*}\mathbb{Q}_{\ell}(p))$.
\\

So the map $P_{p}\rightarrow H^{1-b}_{\text{\'et}}(B,{}^{\mathfrak{p}}\mathcal{H}^{2p-1+b}(\textbf{R}\pi_{*}\mathbb{Q}_{\ell}(p)))$ is induced from
\begin{align*}
{}^{\mathfrak{p}}\tau_{\leq 2p-1+b}\mathbf{R}\pi_{*}\mathbb{Q}_{\ell}(p)\rightarrow {}^{\mathfrak{p}}\mathcal{H}^{2p-1+b}(\textbf{R}\pi_{*}\mathbb{Q}_{\ell}(p))
\end{align*}

Let us momentarily denote $\mathcal{F}_{p}=\mathbf{R}\pi_{*}\mathbb{Q}_{\ell}(p)$ and likewise for $\mathcal{F}_{q}$.
\begin{lemma}
\label{lem:DerSq}
We have the following commutative diagram involving the two derived pairings
\begin{center}
\begin{tikzcd}
  {}^{\mathfrak{p}}\tau_{\leq 2p-1+b}\mathcal{F}_{p}\otimes^{\mathbf{L}} {}^{\mathfrak{p}}\tau_{\leq 2q-1+b}\mathcal{F}_{q} \arrow[d] \arrow[r] 
  & \mathcal{F}_{p}\otimes^{\mathbf{L}} \mathcal{F}_{q} \arrow[d] \\
  {}^{\mathfrak{p}}\mathcal{H}^{2p-1+b}(\mathcal{F}_{p})[-2p+1-b]\otimes^{\mathbf{L}} {}^{\mathfrak{p}}\mathcal{H}^{2q-1+b}(\mathcal{F}_{q})[-2q+1-b] \arrow[r]
  & \mathbb{Q}_{\ell}(1)[-2d]
\end{tikzcd}
\end{center}
\end{lemma}
\begin{proof}
We want to show two pairings
\begin{align*}
{}^{\mathfrak{p}}\tau_{\leq 2p-1+b}\mathcal{F}_{p}\otimes^{\mathbf{L}} {}^{\mathfrak{p}}\tau_{\leq 2q-1+b}\mathcal{F}_{q}\rightarrow \mathbb{Q}_{\ell}(1)[-2d]
\end{align*}
are equal. So let us show the maps
\begin{align*}
{}^{\mathfrak{p}}\tau_{\leq 0}\widetilde{\mathcal{F}_{p}}\rightarrow D({}^{\mathfrak{p}}\tau_{\leq 0}\widetilde{\mathcal{F}_{q}})
\end{align*}
which they correspond to are equal, where $\widetilde{\mathcal{F}_{p}}=\mathcal{F}_{p}[2p-1+b]$ and likewise for $\widetilde{\mathcal{F}_{q}}$.\\

This means showing the following diagram commutes
\begin{center}
\begin{tikzcd}
  {}^{\mathfrak{p}}\mathcal{H}^{0}\widetilde{\mathcal{F}_{p}} \arrow[r]
  & D({}^{\mathfrak{p}}\mathcal{H}^{0}\widetilde{\mathcal{F}_{q}}) \arrow[d] \\
  {}^{\mathfrak{p}}\tau_{\leq 0}\widetilde{\mathcal{F}_{p}} \arrow[r] \arrow[d] \arrow[u]
  & D({}^{\mathfrak{p}}\tau_{\leq 0}\widetilde{\mathcal{F}_{q}}) \\
  \widetilde{\mathcal{F}_{p}} \arrow[r]
  & D(\widetilde{\mathcal{F}_{q}}) \arrow[u]
\end{tikzcd}
\end{center}
Note that ${}^{\mathfrak{p}}\tau_{\leq 0}\widetilde{\mathcal{F}_{p}}\in 
{}^{\mathfrak{p}}D^{\leq 0}$ and $D({}^{\mathfrak{p}}\tau_{\leq 0}\widetilde{\mathcal{F}_{q}})={}^{\mathfrak{p}}\tau_{\geq 0}D(\widetilde{\mathcal{F}_{q}})\in {}^{\mathfrak{p}}D^{\geq 0}$, so the commutativity follows from the standard result for t-structures that
\begin{align*}
\textnormal{Hom}(A,B)\simeq\textnormal{Hom}(H^{0}(A),H^{0}(B))
\end{align*}
for $A\in D^{\leq 0}$ and $B\in D^{\geq 0}$.

\end{proof}
Taking cohomology of the commutative square in the statement of the lemma, the proposition now follows. 

\end{proof}
\begin{proof}[Proof. (of \cref{th:Coh=Geo})]
Using \cref{prop:Coh-Geom} we just need to observe that the image of $\textnormal{cl}(z^{p}_{B})$ under
\begin{align*}
P_{p}\rightarrow H^{1-b}_{\textnormal{\'et}}(B,{}^{\mathfrak{p}}\mathcal{H}^{2p-1+b}(\textbf{R}\pi_{*}\mathbb{Q}_{\ell}(p)))
\end{align*}
is equal to $\alpha^{p}_{B}$ by definition, and likewise for $q$.

\end{proof}
\subsection{Comparing Cohomological Pairings}
\label{sec:CohPair}
Let us now investigate how the two cohomological pairings $h_{B}$ and $h_{U}$ compare.
\begin{theorem}
\label{th:CompCoh}
For any $\alpha^{p}\in CH^{p}_{\textnormal{hom}}(X)_{\mathbb{Q}}$, $\alpha^{q}\in CH^{q}_{\textnormal{hom}}(X)_{\mathbb{Q}}$ we have
\begin{align*}
\langle\alpha^{p},\alpha^{q}\rangle_{B}=\langle\alpha^{p},\alpha^{q}\rangle_{U}+h_{B}(\widehat{\alpha^{p}_{B}},\widehat{\alpha^{q}_{B}})
\end{align*}
Moreover, if we can find a representative $z^{p}_{B}$ such that the image of $\textnormal{cl}(z^{p}_{B})$ inside $H^{0}_{\textnormal{\'et}}(B,\textnormal{\textbf{R}}^{2p}\pi_{*}\mathbb{Q}_{\ell}(p)))$ is zero, then for the corresponding $\alpha^{p}_{B}$ we have $\widehat{\alpha^{p}_{B}}=0$, and consequently
\begin{align*}
\langle\alpha^{p},\alpha^{q}\rangle_{B}=\langle\alpha^{p},\alpha^{q}\rangle_{U}
\end{align*}
\end{theorem}
See \eqref{eq:Decomp} for the meaning of $\widehat{\alpha^{p}_{B}}$. Combining \cref{th:CompCoh} with \cref{th:Coh=Geo} gives the following corollary.
\begin{corollary}
\label{cor:Geo=Perv}
For any $\alpha^{p}\in CH^{p}_{\textnormal{hom}}(X)_{\mathbb{Q}}$, $\alpha^{q}\in CH^{q}_{\textnormal{hom}}(X)_{\mathbb{Q}}$ we have
\begin{align*}
\textnormal{cl}_{B}\langle\alpha^{p},\alpha^{q}\rangle_{\mathcal{X}}=\langle\alpha^{p},\alpha^{q}\rangle_{U}+h_{B}(\widehat{\alpha^{p}_{B}},\widehat{\alpha^{q}_{B}})
\end{align*}
\end{corollary}
Now to prove \cref{th:CompCoh}, we start with the decomposition theorem \eqref{eq:DecompThm}, which gives us a map
\begin{align*}
j_{!*}\mathbf{R}^{2p-1}\pi_{U*}\mathbb{Q}_{\ell}(p)[b]\rightarrow {}^{\mathfrak{p}}\mathcal{H}^{2p-1+b}(\textbf{R}\pi_{*}\mathbb{Q}_{\ell}(p))
\end{align*}
\begin{prop}
\label{prop:CommTri}
The following triangle commutes
\begin{center}
\begin{tikzcd}
  j_{!*}\mathbf{R}^{2p-1}\pi_{U*}\mathbb{Q}_{\ell}(p)[b]\otimes^{\mathbf{L}} j_{!*}\mathbf{R}^{2q-1}\pi_{U*}\mathbb{Q}_{\ell}(q)[b] \arrow[dd] \arrow[dr, "h_{U}"] 
  & ~ \\
  ~
  & \mathbb{Q}_{\ell}(1)[2b] \\  
  {}^{\mathfrak{p}}\mathcal{H}^{2p-1+b}(\textnormal{\textbf{R}}\pi_{*}\mathbb{Q}_{\ell}(p))\otimes^{\mathbf{L}} {}^{\mathfrak{p}}\mathcal{H}^{2q-1+b}(\textnormal{\textbf{R}}\pi_{*}\mathbb{Q}_{\ell}(q)) \arrow[ur, "h_{B}"']
  & ~
\end{tikzcd}
\end{center}
\end{prop}
\begin{proof}
We simply go through the construction of $h_{B}$, using the decomposition theorem to decompose
\begin{align*}
\mathbf{R}\pi_{*}\mathbb{Q}_{\ell}(p)=j_{!*}(\mathbf{R}^{2p-1}\pi_{U*}\mathbb{Q}_{\ell}(p)[b])[-2p+1-b]\oplus C
\end{align*}
and likewise for $q$, where $C$ denotes the spare terms in the decomposition theorem.
\\

Indeed, to construct $h_{B}$ one starts with the map \eqref{eq:DerMap}
\begin{align*}
j_{!*}\mathbf{R}^{2p-1}\pi_{U*}\mathbb{Q}_{\ell}(p)[b]\oplus C=\mathbf{R}\pi_{*}\mathbb{Q}_{\ell}(p)[2p-1+b]&\rightarrow D(\mathbf{R}\pi_{*}\mathbb{Q}_{\ell}(q-1+b)[2q-1+b])\\
&=D(j_{!*}\mathbf{R}^{2q-1}\pi_{U*}\mathbb{Q}_{\ell}(q)[b]\oplus C')\\
&=j_{!*}D(\mathbf{R}^{2q-1}\pi_{U*}\mathbb{Q}_{\ell}(q)[b])\oplus D(C')
\end{align*}
which contains a map $j_{!*}\mathbf{R}^{2p-1}\pi_{U*}\mathbb{Q}_{\ell}(p)[b]\rightarrow j_{!*}D(\mathbf{R}^{2q-1}\pi_{U*}\mathbb{Q}_{\ell}(q)[b])$ that we claim is the one corresponding to $h_{U}$.
\\

To see that this is true, note that if we pullback this map to $U$ we get the duality map $\mathbf{R}^{2p-1}\pi_{U*}\mathbb{Q}_{\ell}(p)[b]\rightarrow D(\mathbf{R}^{2q-1}\pi_{U*}\mathbb{Q}_{\ell}(q)[b])$, from which $h_{U}$ is defined by applying $j_{!*}$, so now we just need to appeal to the continuation principle, see Corollary III.5.11.~of \cite{Kie}.
\\

So when we apply ${}^{\mathfrak{p}}\mathcal{H}^{0}$ to \eqref{eq:DerMap} to get (the map corresponding to) the pairing $h_{B}$, as ${}^{\mathfrak{p}}\mathcal{H}^{0}$ does nothing on $j_{!*}\mathbf{R}^{2p-1}\pi_{U*}\mathbb{Q}_{\ell}(p)[b]$ and $j_{!*}\mathbf{R}^{2q-1}\pi_{U*}\mathbb{Q}_{\ell}(q)[b]$, we see that the pairing $h_{B}$ also contains $h_{U}$, and the proposition follows.

\end{proof}

When we pass to cohomology we get the commutative triangle
\begin{center}
\begin{tikzcd}
  H^{1-b}_{\text{\'et}}(B,j_{!*}\mathbf{R}^{2p-1}\pi_{U*}\mathbb{Q}_{\ell}(p)[b])\times H^{1-b}_{\text{\'et}}(B,j_{!*}\mathbf{R}^{2q-1}\pi_{U*}\mathbb{Q}_{\ell}(q)[b]) \arrow[dd] \arrow[dr, "h_{U}"] 
  & ~ \\
  ~
  & H^{2}_{\text{\'et}}(B,\mathbb{Q}_{\ell}(1)) \\  
  H^{1-b}_{\text{\'et}}(B,{}^{\mathfrak{p}}\mathcal{H}^{2p-1+b}(\textbf{R}\pi_{*}\mathbb{Q}_{\ell}(p)))\times H^{1-b}_{\text{\'et}}(B,{}^{\mathfrak{p}}\mathcal{H}^{2q-1+b}(\textbf{R}\pi_{*}\mathbb{Q}_{\ell}(q))) \arrow[ur, "h_{B}"']
  & ~
\end{tikzcd}
\end{center}
and what remains is to understand the relation between $\alpha^{p}_{U}$ and $\alpha^{p}_{B}$.
\\

To do this, first note that we have the commutative triangle
\begin{center}
\begin{tikzcd}
  j_{!*}\mathbf{R}^{2p-1}\pi_{U*}\mathbb{Q}_{\ell}(p)[b] \arrow[d] \arrow[r] 
  & \mathbf{R}j_{*}\mathbf{R}^{2p-1}\pi_{U*}\mathbb{Q}_{\ell}(p)[b] \\
  {}^{\mathfrak{p}}\mathcal{H}^{2p-1+b}(\textbf{R}\pi_{*}\mathbb{Q}_{\ell}(p)) \arrow[ur]
  & ~
\end{tikzcd}
\end{center}
which when we pass to cohomology gives
\begin{center}
\begin{tikzcd}
  H^{1-b}_{\text{\'et}}(B,j_{!*}\mathbf{R}^{2p-1}\pi_{U*}\mathbb{Q}_{\ell}(p)[b]) \arrow[d, "\iota"'] \arrow[hookrightarrow]{r} 
  & H^{1}_{\text{\'et}}(U,\mathbf{R}^{2p-1}\pi_{U*}\mathbb{Q}_{\ell}(p)) \\
  H^{1-b}_{\text{\'et}}(B,{}^{\mathfrak{p}}\mathcal{H}^{2p-1+b}(\textbf{R}\pi_{*}\mathbb{Q}_{\ell}(p))) \arrow[ur]
  & ~
\end{tikzcd}
\end{center}
Note that the kernel of the map
\begin{align*}
H^{1-b}_{\text{\'et}}(B,j_{!*}\mathbf{R}^{2p-1}\pi_{U*}\mathbb{Q}_{\ell}(p)[b])\oplus H^{1-b}_{\text{\'et}}(B,C_{2p-1})&\simeq H^{1-b}_{\text{\'et}}(B,{}^{\mathfrak{p}}\mathcal{H}^{2p-1+b}(\textbf{R}\pi_{*}\mathbb{Q}_{\ell}(p)))\\
&\rightarrow H^{1}_{\text{\'et}}(U,\mathbf{R}^{2p-1}\pi_{U*}\mathbb{Q}_{\ell}(p))
\end{align*}
is exactly $H^{1-b}_{\text{\'et}}(B,C_{2p-1})$, because $C_{2p-1}$ is supported on $Z$.
\\

Now consider the commutative diagram
\begin{center}
\begin{tikzcd}
  P_{p} \arrow[d] \arrow[r] 
  & H^{1-b}_{\text{\'et}}(B,{}^{\mathfrak{p}}\mathcal{H}^{2p-1+b}(\textbf{R}\pi_{*}\mathbb{Q}_{\ell}(p))) \arrow[d] \\ 
   P^{'}_{p} \arrow[r]
  & H^{1}_{\text{\'et}}(U,\mathbf{R}^{2p-1}\pi_{U*}\mathbb{Q}_{\ell}(p))
\end{tikzcd}
\end{center}
where $P^{'}_{p}=\text{Ker}(H^{2p}_{\text{\'et}}(\mathcal{X}_{U},\mathbb{Q}_{\ell}(p))\rightarrow H^{1}_{\text{\'et}}(U,\mathbf{R}^{2p-1}\pi_{U*}\mathbb{Q}_{\ell}(p)))$. The commutativity of the diagram follows from the map \eqref{eq:SpecSeq} of spectral sequences
\\

If we trace the image of $\text{cl}(z^{p}_{B})$ anticlockwise we get $\text{cl}(z^{p}_{U})$ and then $\alpha^{p}_{U}$. On the other hand, if we go clockwise we get $\alpha^{p}_{B}$ and then whatever it maps to, which therefore must be $\alpha^{p}_{U}$.
\\

So if we write
\begin{equation}
\label{eq:Decomp}
\alpha^{p}_{B}=\widetilde{\alpha^{p}_{B}}+\widehat{\alpha^{p}_{B}}\in H^{1-b}_{\text{\'et}}(B,j_{!*}\mathbf{R}^{2p-1}\pi_{U*}\mathbb{Q}_{\ell}(p)[b])\oplus H^{1-b}_{\text{\'et}}(B,C_{2p-1})
\end{equation}
then we must have $\widetilde{\alpha^{p}_{B}}=\iota(\alpha^{p}_{U})$.
\\

We have then that
\begin{align*}
\langle \alpha^{p},\alpha^{q}\rangle_{B}&=h_{B}(\widetilde{\alpha^{p}_{B}},\widetilde{\alpha^{q}_{B}})+h_{B}(\widetilde{\alpha^{p}_{B}},\widehat{\alpha^{q}_{B}})+h_{B}(\widehat{\alpha^{p}_{B}},\widetilde{\alpha^{q}_{B}})+h_{B}(\widehat{\alpha^{p}_{B}},\widehat{\alpha^{q}_{B}})\\
&=\langle \alpha^{p},\alpha^{q}\rangle_{U}+h_{B}(\widetilde{\alpha^{p}_{B}},\widehat{\alpha^{q}_{B}})+h_{B}(\widehat{\alpha^{p}_{B}},\widetilde{\alpha^{q}_{B}})+h_{B}(\widehat{\alpha^{p}_{B}},\widehat{\alpha^{q}_{B}})
\end{align*}
\begin{lemma}
\label{lem:Orthog}
Both of the cross-terms are zero
\begin{align*}
h_{B}(\widetilde{\alpha^{p}_{B}},\widehat{\alpha^{q}_{B}})=0=h_{B}(\widehat{\alpha^{p}_{B}},\widetilde{\alpha^{q}_{B}})
\end{align*}
\end{lemma}
\begin{proof}
Notice that the pairing $h_{B}$ corresponds to the map
\begin{align*}
{}^{\mathfrak{p}}\mathcal{H}^{2p-1+b}(\mathbf{R}\pi_{*}\mathbb{Q}_{\ell}(p))\rightarrow D({}^{\mathfrak{p}}\mathcal{H}^{2q-1+b}(\mathbf{R}\pi_{*}\mathbb{Q}_{\ell}(q)))
\end{align*}
Using the decomposition theorem this splits as a direct sum of maps, and the cross-terms come from the following respective parts of this decomposition
\begin{align*}
j_{!*}\mathbf{R}^{2p-1}\pi_{U*}\mathbb{Q}_{\ell}(p)[b]\rightarrow i_{*}D(\widehat{C}_{2q-1}')\\[5pt]
i_{*}\widehat{C}_{2p-1}\rightarrow j_{!*}D(\mathbf{R}^{2q-1}\pi_{U*}\mathbb{Q}_{\ell}(q)[b])
\end{align*}
where $i_{*}\widehat{C}_{2p-1}=C_{2p-1}$, $i_{*}\widehat{C}_{2q-1}'=C_{2q-1}'$ are the terms supported on $Z$.\\

We claim that in general for $A\in\text{Perv}(U)$, $B\in\text{Perv}(Z)$
\begin{align*}
\textnormal{Hom}(j_{!*}A,i_{*}B)=0=\textnormal{Hom}(i_{*}B,j_{!*}A)
\end{align*}
from which the lemma follows. The two statements are dual, so we just prove the first. Indeed, we have
\begin{align*}
\textnormal{Hom}(j_{!*}A,i_{*}B)\simeq\textnormal{Hom}(i^{*}j_{!*}A,B)
\end{align*}
where $i^{*}j_{!*}A\in {}^{\mathfrak{p}}D^{\leq -1}(Z)$ by Lemma III.5.1.~of \cite{Kie}, and $B\in {}^{\mathfrak{p}}D^{\geq 0}(Z)$ since $B$ is perverse. Hence,
\begin{align*}
\textnormal{Hom}(i^{*}j_{!*}A,B)=0
\end{align*}
and so we are done.

\end{proof}
It follows from the lemma that
\begin{align*}
\langle \alpha^{p},\alpha^{q}\rangle_{B}=\langle \alpha^{p},\alpha^{q}\rangle_{U}+h_{B}(\widehat{\alpha^{p}_{B}},\widehat{\alpha^{q}_{B}})
\end{align*}
In general the spare term $h_{B}(\widehat{\alpha^{p}_{B}},\widehat{\alpha^{q}_{B}})$ will be non-zero, but if we choose our extension of $\alpha^{p}$ appropriately then we claim that the class $\widehat{\alpha^{p}_{B}}$ will be zero.
\begin{theorem}
\label{th:BeilConj}
Suppose that the cycle $z^{p}_{B}$ on $\mathcal{X}$ is such that the image of $\textnormal{cl}(z^{p}_{B})$ inside $H^{0}_{\textnormal{\'et}}(B,\textnormal{\textbf{R}}^{2p}\pi_{*}\mathbb{Q}_{\ell}(p)))$ is zero, then $\widehat{\alpha^{p}_{B}}=0$.
\end{theorem}
In the above situation then, we have
\begin{align*}
\langle \alpha^{p},\alpha^{q}\rangle_{B}=\langle \alpha^{p},\alpha^{q}\rangle_{U}
\end{align*}
and this completes the proof of \cref{th:CompCoh}. We conjecture that such a cycle like this can always be found.
\begin{conj}
\label{conj:BeilCyc}
For any $\alpha^{p}\in CH^{p}_{\textnormal{hom}}(X)_{\mathbb{Q}}$, we can find an extension $z^{p}_{B}$ on $\mathcal{X}$ such that the image of $\textnormal{cl}(z^{p}_{B})$ inside $H^{0}_{\textnormal{\'et}}(B,\textnormal{\textbf{R}}^{2p}\pi_{*}\mathbb{Q}_{\ell}(p)))$ is zero.
\end{conj}
Let us now prove \cref{th:BeilConj}.
\begin{proof}[Proof. (of \cref{th:BeilConj})]
We start by noting that if we pull the decomposition \eqref{eq:DecompThm} back to $Z$ we get
\begin{align}
\label{eq:DecompZ}
\mathbf{R}\pi_{Z*}\mathbb{Q}_{\ell}(p)\simeq\bigoplus_{k\geq 0} \big(i^{*}j_{!*}(\mathbf{R}^{k}\pi_{U*}\mathbb{Q}_{\ell}(p)[b])\oplus \widehat{C}_{k}\big)[-k-b]
\end{align}
Since $j$ is affine, we have that $i^{*}j_{!*}(\mathbf{R}^{k}\pi_{U*}\mathbb{Q}_{\ell}(p)[b])[-1]$ is a perverse sheaf on $Z$ (see the discussion after Corollary III.6.2.~in \cite{Kie}), and therefore
\begin{align*}
{}^{\mathfrak{p}}\mathcal{H}^{2p-1+b}(\mathbf{R}\pi_{Z*}\mathbb{Q}_{\ell}(p))=i^{*}j_{!*}(\mathbf{R}^{2p}\pi_{U*}\mathbb{Q}_{\ell}(p)[b])[-1]\oplus \widehat{C}_{2p-1}
\end{align*}
Now, we have the following commutative diagram
\begin{center}
\begin{tikzcd}
  \mathbf{R}\pi_{*}\mathbb{Q}_{\ell}(p) \arrow[d] \arrow[twoheadrightarrow]{r} 
  & {}^{\mathfrak{p}}\mathcal{H}^{2p-1+b}(\textbf{R}\pi_{*}\mathbb{Q}_{\ell}(p))[-2p+1-b] \arrow[twoheadrightarrow]{r}
  & C_{2p-1}[-2p+1-b] \arrow[d, "\sim"]\\ 
   i_{*}\mathbf{R}\pi_{Z*}\mathbb{Q}_{\ell}(p) \arrow[twoheadrightarrow]{r}
   & i_{*}{}^{\mathfrak{p}}\mathcal{H}^{2p-1+b}(\textbf{R}\pi_{Z*}\mathbb{Q}_{\ell}(p))[-2p+1-b] \arrow[twoheadrightarrow]{r}
  & i_{*}\widehat{C}_{2p-1}[-2p+1-b]
\end{tikzcd}
\end{center}
where the horizontal arrows are the projection maps from the respective decompositions \eqref{eq:DecompThm} and \eqref{eq:DecompZ}.\\

When we take cohomology of the above diagram we get
\begin{center}
\begin{tikzcd}
  H^{2p}_{\text{\'et}}(\mathcal{X},\mathbb{Q}_{\ell}(p)) \arrow[d] \arrow[r] 
  & H^{1-b}_{\text{\'et}}(B,{}^{\mathfrak{p}}\mathcal{H}^{2p-1+b}(\textbf{R}\pi_{*}\mathbb{Q}_{\ell}(p))) \arrow[twoheadrightarrow]{r}
  & H^{1-b}_{\text{\'et}}(B,C_{2p-1}) \arrow[d, "\sim"]\\ 
   H^{2p}_{\text{\'et}}(\mathcal{X}_{Z},\mathbb{Q}_{\ell}(p)) \arrow[r]
  & H^{1-b}_{\text{\'et}}(Z,{}^{\mathfrak{p}}\mathcal{H}^{2p-1+b}(\textbf{R}\pi_{Z*}\mathbb{Q}_{\ell}(p))) \arrow[twoheadrightarrow]{r}
  & H^{1-b}_{\text{\'et}}(Z,\widehat{C}_{2p-1})
\end{tikzcd}
\end{center}
which tells us that $\widehat{\alpha^{p}_{B}}$ is equal to the image of $\text{cl}(z^{p}_{B})$ as we go round the diagram anticlockwise.\\

So to prove the theorem, it suffices to prove that the image of $\text{cl}(z^{p}_{B})$ in \\$H^{1-b}_{\text{\'et}}(Z,{}^{\mathfrak{p}}\mathcal{H}^{2p-1+b}(\textbf{R}\pi_{Z*}\mathbb{Q}_{\ell}(p)))$ is zero. By assumption on $z^{p}_{B}$ we know that the image in $H^{0}_{\text{\'et}}(B,\textbf{R}^{2p}\pi_{*}\mathbb{Q}_{\ell}(p)))$ is zero. We claim that there is a map fitting into the following commutative diagram
\begin{center}
\begin{equation}
\label{eq:CommTri}
\begin{tikzcd}
H^{2p}_{\text{\'et}}(\mathcal{X},\mathbb{Q}_{\ell}(p)) \arrow[d] \arrow[r] 
  & H^{0}_{\text{\'et}}(B,\textbf{R}^{2p}\pi_{*}\mathbb{Q}_{\ell}(p))) \arrow[d] \\
  H^{2p}_{\text{\'et}}(\mathcal{X}_{Z},\mathbb{Q}_{\ell}(p)) \arrow[dr] \arrow[r] 
  & H^{0}_{\text{\'et}}(Z,\textbf{R}^{2p}\pi_{Z*}\mathbb{Q}_{\ell}(p))) \arrow[dashrightarrow]{d} \\ 
   ~
   & H^{1-b}_{\text{\'et}}(Z,{}^{\mathfrak{p}}\mathcal{H}^{2p-1+b}(\textbf{R}\pi_{Z*}\mathbb{Q}_{\ell}(p)))
\end{tikzcd}
\end{equation}
\end{center}
from which the theorem would then follow.
\\

To construct said map, first recall that for any $K$ we have an exact triangle
\begin{align*}
\tau_{\leq 2p}K\rightarrow K\rightarrow \tau_{\geq 2p+1}K\rightarrow \tau_{\leq 2p}K[1]
\end{align*}
Applying this to $\tau_{\geq 2p}\mathcal{F}$ where $\mathcal{F}=\textbf{R}\pi_{Z*}\mathbb{Q}_{\ell}(p)$ gives us the exact triangle
\begin{align*}
\textbf{R}^{2p}\pi_{Z*}\mathbb{Q}_{\ell}(p)[-2p]\rightarrow \tau_{\geq 2p}\mathcal{F}\rightarrow \tau_{\geq 2p+1}\mathcal{F}\rightarrow \textbf{R}^{2p}\pi_{Z*}\mathbb{Q}_{\ell}(p)[-2p+1]
\end{align*}
The perverse long exact sequence induces a map
\begin{align}
\label{eq:BoundaryMap}
{}^{\mathfrak{p}}\mathcal{H}^{0}(\textbf{R}^{2p}\pi_{Z*}\mathbb{Q}_{\ell}(p)[b-1])\rightarrow {}^{\mathfrak{p}}\mathcal{H}^{2p-1+b}(\tau_{\geq 2p}\textbf{R}\pi_{Z*}\mathbb{Q}_{\ell}(p))
\end{align}
Next, we use the exact triangle
\begin{align*}
\tau_{\leq 2p-1}\mathcal{F}\rightarrow \mathcal{F}\rightarrow \tau_{\geq 2p}\mathcal{F}\rightarrow \tau_{\leq 2p-1}\mathcal{F}[1]
\end{align*}
and the perverse long exact sequence gives
\begin{align*}
{}^{\mathfrak{p}}\mathcal{H}^{2p-1+b}(\tau_{\leq 2p-1}\mathcal{F})\rightarrow {}^{\mathfrak{p}}\mathcal{H}^{2p-1+b}(\mathcal{F})\rightarrow {}^{\mathfrak{p}}\mathcal{H}^{2p-1+b}(\tau_{\geq 2p}\mathcal{F})\rightarrow {}^{\mathfrak{p}}\mathcal{H}^{2p+b}(\tau_{\leq 2p-1}\mathcal{F})
\end{align*}
\begin{lemma}
\label{lem:BoundCoh}
We have that
\begin{align*}
{}^{\mathfrak{p}}\mathcal{H}^{c}(\tau_{\leq 2p-1}\mathcal{F})=0
\end{align*}
for all $c>2p-2+b$.
\end{lemma}
\begin{proof}
We use the exact triangle
\begin{align*}
\label{eq:ExactSeq}
\tau_{\leq 2p-2}\mathcal{F}\rightarrow \tau_{\leq 2p-1}\mathcal{F}\rightarrow \textbf{R}^{2p-1}\pi_{Z*}\mathbb{Q}_{\ell}(p)[-2p+1]\rightarrow \tau_{\leq 2p-2}\mathcal{F}[1]
\end{align*}
The perverse long exact sequence gives us the exact sequence
\begin{align}
{}^{\mathfrak{p}}\mathcal{H}^{c}(\tau_{\leq 2p-2}\mathcal{F})\rightarrow {}^{\mathfrak{p}}\mathcal{H}^{c}(\tau_{\leq 2p-1}\mathcal{F})\rightarrow {}^{\mathfrak{p}}\mathcal{H}^{c-2p+1}(\textbf{R}^{2p-1}\pi_{Z*}\mathbb{Q}_{\ell}(p))
\end{align}
Note that
\begin{align*}
{}^{\mathfrak{p}}\mathcal{H}^{c-2p+1}(\textbf{R}^{2p-1}\pi_{Z*}\mathbb{Q}_{\ell}(p))={}^{\mathfrak{p}}\mathcal{H}^{c-2p+2-b}(\textbf{R}^{2p-1}\pi_{Z*}\mathbb{Q}_{\ell}(p)[b-1])
\end{align*}
and we have the following sub-lemma
\begin{lemma}
\label{lem:BoundCohZ}
For any sheaf $\mathcal{G}$ on $Z$ we have
\begin{align*}
{}^{\mathfrak{p}}\mathcal{H}^{c}(\mathcal{G}[b-1])=0
\end{align*}
for all $c>0$.
\end{lemma}
\begin{proof}
We claim that $\tau_{\geq 0}\mathcal{G}[b]=0$, which is enough to prove the lemma.
\\

Indeed, for any $A\in {}^{\mathfrak{p}}D^{\geq 0}(Z)$ we know that $A\in D^{\geq 1-b}(Z)$ by Lemma III.5.13.~of \cite{Kie}, and since obviously $\mathcal{G}[b]\in D^{\leq -b}(Z)$, it follows that
\begin{align*}
\textnormal{Hom}(\mathcal{G}[b],A)=0
\end{align*}
which implies that $\tau_{\geq 0}\mathcal{G}[b]=0$ as desired.

\end{proof}
So now, to prove \cref{lem:BoundCoh} using \eqref{eq:ExactSeq}; by \cref{lem:BoundCohZ} the RHS group is zero for all $c>2p-2+b$ and thus it just remains to show that the LHS group ${}^{\mathfrak{p}}\mathcal{H}^{c}(\tau_{\leq 2p-2}\mathcal{F})$ is also zero for such $c$.
\\

To do this one simply uses the same argument, replacing $2p-1$ with $2p-2$, and so on until we find $n$ large enough such that $\tau_{\leq 2p-n}\mathcal{F}=0$.

\end{proof}
By \cref{lem:BoundCoh} then, we have an isomorphism
\begin{align*}
{}^{\mathfrak{p}}\mathcal{H}^{2p-1+b}(\mathcal{F})\xrightarrow{\sim} {}^{\mathfrak{p}}\mathcal{H}^{2p-1+b}(\tau_{\geq 2p}\mathcal{F})
\end{align*}
which if we compose with \eqref{eq:BoundaryMap} gives us a map
\begin{align*}
{}^{\mathfrak{p}}\mathcal{H}^{0}(\textbf{R}^{2p}\pi_{Z*}\mathbb{Q}_{\ell}(p)[b-1])\rightarrow {}^{\mathfrak{p}}\mathcal{H}^{2p-1+b}(\textbf{R}\pi_{Z*}\mathbb{Q}_{\ell}(p))
\end{align*}
Taking cohomology gives
\begin{align*}
H^{1-b}_{\text{\'et}}(Z,{}^{\mathfrak{p}}\mathcal{H}^{0}(\textbf{R}^{2p}\pi_{Z*}\mathbb{Q}_{\ell}(p)[b-1]))\rightarrow H^{1-b}_{\text{\'et}}(Z,{}^{\mathfrak{p}}\mathcal{H}^{2p-1+b}(\textbf{R}\pi_{Z*}\mathbb{Q}_{\ell}(p)))
\end{align*}
which if we compose with the spectral sequence map
\begin{align*}
H^{0}_{\text{\'et}}(Z,\textbf{R}^{2p}\pi_{Z*}\mathbb{Q}_{\ell}(p))\rightarrow H^{1-b}_{\text{\'et}}(Z,{}^{\mathfrak{p}}\mathcal{H}^{0}(\textbf{R}^{2p}\pi_{Z*}\mathbb{Q}_{\ell}(p)[b-1]))
\end{align*}
gives us the desired map
\begin{align}
\label{eq:PervMap}
H^{0}_{\text{\'et}}(Z,\textbf{R}^{2p}\pi_{Z*}\mathbb{Q}_{\ell}(p))\rightarrow H^{1-b}_{\text{\'et}}(Z,{}^{\mathfrak{p}}\mathcal{H}^{2p-1+b}(\textbf{R}\pi_{Z*}\mathbb{Q}_{\ell}(p)))
\end{align}
Finally, the commutativity of \eqref{eq:CommTri} results from the following commutative diagram which fits all the spectral sequence maps together
\begin{center}
\begin{tikzcd}
  H^{0}_{\text{\'et}}(Z,\textbf{R}^{2p}\pi_{Z*}\mathbb{Q}_{\ell}(p)) \arrow[r] \arrow[d]
  & H^{1-b}_{\text{\'et}}(Z,{}^{\mathfrak{p}}\mathcal{H}^{0}(\textbf{R}^{2p}\pi_{Z*}\mathbb{Q}_{\ell}(p)[b-1])) \arrow[d] \\
  H^{2p}_{\text{\'et}}(Z,\tau_{\geq 2p}\textbf{R}\pi_{Z*}\mathbb{Q}_{\ell}(p)) \arrow[r]
  & H^{1-b}_{\text{\'et}}(Z,{}^{\mathfrak{p}}\mathcal{H}^{2p-1+b}(\tau_{\geq 2p}\textbf{R}\pi_{Z*}\mathbb{Q}_{\ell}(p))) \\
  H^{2p}_{\text{\'et}}(\mathcal{X}_{Z},\mathbb{Q}_{\ell}(p)) \arrow[r] \arrow[u] \arrow[bend left=80]{uu}
  & H^{1-b}_{\text{\'et}}(Z,{}^{\mathfrak{p}}\mathcal{H}^{2p-1+b}(\textbf{R}\pi_{Z*}\mathbb{Q}_{\ell}(p))) \arrow["\sim"']{u}
\end{tikzcd}
\end{center}
Observe that the map \eqref{eq:PervMap} is obtained by going clockwise around the diagram. This completes the proof of \cref{th:BeilConj}.

\end{proof}
\section{The Projection Formula}
\label{sec:ProjForm}
From this point on, let us denote Beilinson's pairing as $\langle -,-\rangle_{X}$ (recall that we previously denoted it by $\langle -,-\rangle_{U}$), and likewise the pairing $h_{U}$ as $h_{X}$.\\

Suppose $f:B'\rightarrow B$ is a dominant proper morphism of relative dimension $0$ which is generically finite, and let $K'\rightarrow K$ be the induced map on function fields. Then take a finite map $g:X'\rightarrow X$ fitting into the following commutative diagram
\begin{center}
\begin{tikzcd}
  X' \arrow[r, "g"] \arrow[d]
  & X \arrow[d] \\
  K' \arrow[r]
  & K
\end{tikzcd}
\end{center}
\begin{theorem}[Projection Formula]
\label{th:ProjForm}
For any $\alpha^{p}\in CH_{\textnormal{hom}}^{p}(X')_{\mathbb{Q}}$ and $\beta^{q}\in CH_{\textnormal{hom}}^{q}(X)_{\mathbb{Q}}$ the projection formula holds
\begin{align*}
\langle g_{*}\alpha^{p},\beta^{q}\rangle_{X}=f_{*}\langle \alpha^{p},g^{*}\beta^{q}\rangle_{X'}
\end{align*}
\end{theorem}
This is analogous to the projection formula for the intersection product. Indeed, if we could spread $g$ out to a finite morphism $g_{\mathcal{X}}$ over $B$
\begin{center}
\begin{tikzcd}
  \mathcal{X}' \arrow[r, "g_{\mathcal{X}}"] \arrow[d]
  & \mathcal{X} \arrow[d] \\
  B' \arrow[r, "f"']
  & B
\end{tikzcd}
\end{center}
as-well as find Beilinson cycles $\widetilde{\alpha^{p}}$, $\widetilde{\beta^{q}}$ extending $\alpha^{p}$, $\beta^{q}$, we could immediately deduce the above theorem:
\begin{align*}
\langle g_{*}\alpha^{p},\beta^{q}\rangle_{X}=g_{\mathcal{X}*}\widetilde{\alpha^{p}}\cap \widetilde{\beta^{q}}=f_{*}(\widetilde{\alpha^{p}}\cap g^{*}_{\mathcal{X}}\widetilde{\beta^{q}})=f_{*}\langle \alpha^{p},g^{*}\beta^{q}\rangle_{X'}
\end{align*}
In lieu of such extensions, we shall prove the theorem from first principles using the analogous cohomological result.\\

We can find some dense open affine subset $j:U\hookrightarrow B$ such that $g$ spreads out to a finite morphism $g_{\mathcal{X}}$ over $U$
\begin{center}
\begin{tikzcd}
  \mathcal{X}' \arrow[r, "g_{\mathcal{X}}"] \arrow[d, "\pi'"']
  & \mathcal{X} \arrow[d, "\pi"] \\
  f^{-1}(U) \arrow[r, "f_{U}"']
  & U
\end{tikzcd}
\end{center}
and for which: $f_{U}$ is finite and $\mathcal{X}'\rightarrow f^{-1}(U)$, $\mathcal{X}\rightarrow U$ are smooth projective models. Let us also denote $j':U'=f^{-1}(U)\hookrightarrow B'$.
\begin{prop}
\label{prop:ProjForm}
For any $\alpha\in H^{1-b}_{\textnormal{\'et}}(B',j_{!*}^{'}\textnormal{\textbf{{R}}}^{2p-1}\pi_{*}^{'}\mathbb{Q}_{\ell}(p)[b])$ and \newline$\beta\in H^{1-b}_{\textnormal{\'et}}(B,j_{!*}\textnormal{\textbf{R}}^{2q-1}\pi_{*}\mathbb{Q}_{\ell}(q)[b])$ the projection formula holds
\begin{align*}
h_{X}(g_{*}\alpha,\beta)=f_{*}h_{X'}(\alpha,g^{*}\beta)
\end{align*}
\end{prop}
\cref{th:ProjForm} then follows as a corollary, simply by definition of the height pairing. We have not yet specified how to define $g_{*}\alpha$ and $g^{*}\beta$; we shall define these pushforward and pullback maps in the next section.
\subsection{Constructing the Pushforward and Pullback}
Consider the derived pushforward
\begin{align*}
f_{*}j_{!*}^{'}\textbf{R}^{2p-1}\pi_{*}^{'}\mathbb{Q}_{\ell}(p)[b]
\end{align*}
In what follows we drop the Tate twist for simplicity. Now, let us restrict this complex to $U$:
\begin{align*}
j^{*}f_{*}j_{!*}^{'}\textbf{R}^{2p-1}\pi_{*}^{'}\mathbb{Q}_{\ell}[b]&\simeq f_{U*}j^{'*}j_{!*}^{'}\textbf{R}^{2p-1}\pi_{*}^{'}\mathbb{Q}_{\ell}[b]\\
&\simeq f_{U*}\textbf{R}^{2p-1}\pi_{*}^{'}\mathbb{Q}_{\ell}[b]\\
&\simeq \textbf{R}^{2p-1}\pi_{*}g_{\mathcal{X}*}\mathbb{Q}_{\ell}[b]
\end{align*}
where the last line follows from the fact that $g_{\mathcal{X}*}$ and $f_{U*}$ are exact (since $g_{\mathcal{X}}$ and $f_{U}$ are both finite).
\\

By the decomposition theorem we have a (non-canonical) isomorphism
\begin{align}
\label{eq:PushDecomp}
f_{*}j_{!*}^{'}\textbf{R}^{2p-1}\pi_{*}^{'}\mathbb{Q}_{\ell}[b]\simeq j_{!*}\textbf{R}^{2p-1}\pi_{*}g_{\mathcal{X}*}\mathbb{Q}_{\ell}[b]\oplus i_{*}C 
\end{align}
where $i_{*}C$ is supported on $Z=B\setminus U$.
\\

We want to define pushforward and pullback maps between $f_{*}j_{!*}^{'}\textbf{R}^{2p-1}\pi_{*}^{'}\mathbb{Q}_{\ell}(p)[b]$ and $j_{!*}\textbf{R}^{2p-1}\pi_{*}\mathbb{Q}_{\ell}(p)[b]$ which make the following diagrams commute
\begin{center}
\begin{tikzcd}
  CH_{\text{hom}}^{p}(\mathcal{X}')_{\mathbb{Q}} \arrow[bend left=7.5, "g_{\mathcal{X}*}"]{r} \arrow[d]
  & CH_{\text{hom}}^{p}(\mathcal{X})_{\mathbb{Q}} \arrow[d] \arrow[bend left=7.5, "g_{\mathcal{X}}^{*}"]{l} \\
  H^{2p}_{\text{\'et}}(\mathcal{X}',\mathbb{Q}_{\ell}) \arrow[bend left=7.5]{r} \arrow[d]
  & H^{2p}_{\text{\'et}}(\mathcal{X},\mathbb{Q}_{\ell}) \arrow[d]\arrow[bend left=7.5]{l} \\
  H^{1}_{\text{\'et}}(U',\textbf{R}^{2p-1}\pi_{*}^{'}\mathbb{Q}_{\ell}) \arrow[bend left=7.5]{r}
  & H^{1}_{\text{\'et}}(U,\textbf{R}^{2p-1}\pi_{*}\mathbb{Q}_{\ell}) \arrow[bend left=7.5]{l}\\
  H^{1-b}_{\text{\'et}}(B',j_{!*}^{'}\textbf{R}^{2p-1}\pi_{*}^{'}\mathbb{Q}_{\ell}[b]) \arrow[bend left=7.5, "g_{*}"]{r} \arrow[hookrightarrow]{u}
  & H^{1-b}_{\text{\'et}}(B,j_{!*}\textbf{R}^{2p-1}\pi_{*}\mathbb{Q}_{\ell}[b]) \arrow[hookrightarrow]{u} \arrow[bend left=7.5, "g^{*}"]{l}
\end{tikzcd}
\end{center}
The pushforward and pullback maps on the middle two rows are defined easily using the adjunction maps $g_{\mathcal{X}*}\mathbb{Q}_{\ell}=g_{\mathcal{X}*}g_{\mathcal{X}}^{!}\mathbb{Q}_{\ell}\rightarrow\mathbb{Q}_{\ell}$ and $\mathbb{Q}_{\ell}\rightarrow g_{\mathcal{X}*}g_{\mathcal{X}}^{*}\mathbb{Q}_{\ell}=g_{\mathcal{X}*}\mathbb{Q}_{\ell}$ respectively.
\\

Since the bottom vertical maps are inclusions, this determines the pushforward and pullback maps on the bottom row, however we wish to understand these maps in the derived category. We seek maps that fit into the following commutative diagram
\begin{equation}
\begin{tikzcd}
\label{eq:CommDig}
  f_{*}j_{*}^{'}\textbf{R}^{2p-1}\pi_{*}^{'}\mathbb{Q}_{\ell}[b] \arrow[bend left=7.5, "\sim"]{r}
  & j_{*}f_{U*}\textbf{R}^{2p-1}\pi_{*}^{'}\mathbb{Q}_{\ell}[b] \arrow[bend left=7.5]{r} \arrow[bend left=7.5, "\sim"]{l}
  & j_{*}\textbf{R}^{2p-1}\pi_{*}\mathbb{Q}_{\ell}[b] \arrow[bend left=7.5]{l}  \\
  f_{*}j_{!*}^{'}\textbf{R}^{2p-1}\pi_{*}^{'}\mathbb{Q}_{\ell}[b] \arrow[dashrightarrow, bend left=7.5]{r} \arrow[u]
  & j_{!*}f_{U*}\textbf{R}^{2p-1}\pi_{*}^{'}\mathbb{Q}_{\ell}[b] \arrow[dashrightarrow, bend left=7.5]{l} \arrow[bend left=7.5]{r} \arrow[u]
  & j_{!*}\textbf{R}^{2p-1}\pi_{*}\mathbb{Q}_{\ell}[b] \arrow[bend left=7.5]{l} \arrow[u] 
\end{tikzcd}
\end{equation}
where the vertical maps are induced by $j_{!*}\rightarrow j_{*}$ and $j_{!*}^{'}\rightarrow j_{*}^{'}$, while the rightmost horizontal arrows are again determined by the adjunction maps mentioned above.
\\

To see that these dashed maps exist; applying the isomorphism \eqref{eq:PushDecomp} we have
\begin{align*}
&\text{Hom}(f_{*}j_{!*}^{'}\textbf{R}^{2p-1}\pi_{*}^{'}\mathbb{Q}_{\ell}[b], j_{*}f_{U*}\textbf{R}^{2p-1}\pi_{*}^{'}\mathbb{Q}_{\ell}[b])\\
\simeq{} &\text{Hom}(j_{!*}f_{U*}\textbf{R}^{2p-1}\pi_{*}^{'}\mathbb{Q}_{\ell}[b]\oplus i_{*}C, j_{*}f_{U*}\textbf{R}^{2p-1}\pi_{*}^{'}\mathbb{Q}_{\ell}[b])\\
\simeq{} &\text{Hom}(j_{!*}f_{U*}\textbf{R}^{2p-1}\pi_{*}^{'}\mathbb{Q}_{\ell}[b], j_{*}f_{U*}\textbf{R}^{2p-1}\pi_{*}^{'}\mathbb{Q}_{\ell}[b])\oplus \text{Hom}(i_{*}C, j_{*}f_{U*}\textbf{R}^{2p-1}\pi_{*}^{'}\mathbb{Q}_{\ell}[b])\\
\simeq{} &\text{Hom}(j^{*}j_{!*}f_{U*}\textbf{R}^{2p-1}\pi_{*}^{'}\mathbb{Q}_{\ell}[b],f_{U*}\textbf{R}^{2p-1}\pi_{*}^{'}\mathbb{Q}_{\ell}[b])\oplus \text{Hom}(j^{*}i_{*}C,f_{U*}\textbf{R}^{2p-1}\pi_{*}^{'}\mathbb{Q}_{\ell}[b])\\
\simeq{} &\text{Hom}(f_{U*}\textbf{R}^{2p-1}\pi_{*}^{'}\mathbb{Q}_{\ell}[b],f_{U*}\textbf{R}^{2p-1}\pi_{*}^{'}\mathbb{Q}_{\ell}[b])
\end{align*}
using the adjunction $j^{*}\dashv j_{*}$ and the fact $j^{*}j_{!*}=\text{id}$, $j^{*}i_{*}=0$.
\\

Call $A=f_{U*}\textbf{R}^{2p-1}\pi_{*}^{'}\mathbb{Q}_{\ell}[b]$, then we see that the map $f_{*}j_{!*}^{'}\textbf{R}^{2p-1}\pi_{*}^{'}\mathbb{Q}_{\ell}[b]\rightarrow f_{*}j_{*}^{'}\textbf{R}^{2p-1}\pi_{*}^{'}\mathbb{Q}_{\ell}[b]\xrightarrow{\sim}j_{*}A$ in \eqref{eq:CommDig} is determined by some $h\in \text{End}(A)$, and in fact $h$ is an automorphism since $h$ is obtained by simply restricting to $U$ (on which the above map is an isomorphism).
\\

Obviously the map $j_{!*}A\rightarrow j_{*}A$ is also determined by an element of $\text{End}(A)$, namely $\text{id}_{A}$. So we can define the map
\begin{align*}
f_{*}j_{!*}^{'}\textbf{R}^{2p-1}\pi_{*}^{'}\mathbb{Q}_{\ell}[b]\simeq j_{!*}A\oplus i_{*}C\rightarrow j_{!*}A
\end{align*}
by $(j_{!*}h,0)$ and then the required diagram commutes. Furthermore, the map in the opposite direction
\begin{align*}
j_{!*}A\rightarrow j_{!*}A\oplus i_{*}C\simeq f_{*}j_{!*}^{'}\textbf{R}^{2p-1}\pi_{*}^{'}\mathbb{Q}_{\ell}[b]
\end{align*}
we take to be $(j_{!*}h^{-1},0)$, and again the required diagram in \eqref{eq:CommDig} commutes.
\\

In order to prove the projection formula, we will also want to know that the following diagram commutes
\begin{center}
\begin{equation}
\begin{tikzcd}
\label{eq:DualDiag}
  f_{*}j_{!*}^{'}\textbf{R}^{2p-1}\pi_{*}^{'}\mathbb{Q}_{\ell}[b] \arrow[dashrightarrow, bend left=7.5]{r}
  & j_{!*}f_{U*}\textbf{R}^{2p-1}\pi_{*}^{'}\mathbb{Q}_{\ell}[b] \arrow[dashrightarrow, bend left=7.5]{l}\\
  f_{*}j_{!}^{'}\textbf{R}^{2p-1}\pi_{*}^{'}\mathbb{Q}_{\ell}[b] \arrow[bend left=7.5, "\sim"]{r} \arrow[u]
  & j_{!}f_{U*}\textbf{R}^{2p-1}\pi_{*}^{'}\mathbb{Q}_{\ell}[b] \arrow[bend left=7.5, "\sim"]{l}  \arrow[u] 
\end{tikzcd}
\end{equation}
\end{center}
where the vertical maps are now determined by $j_{!}\rightarrow j_{!*}$, $j_{!}^{'}\rightarrow j_{!*}^{'}$, and the dashed arrows are the ones just constructed.
\\

By an analogous argument, one sees that maps making this diagram commute are determined by $h'\in \text{End}(A)$, so we just need to prove that $h=h'$.
\\

For this, consider the full diagram
 \begin{center}
 \begin{tikzcd}
  f_{*}j_{*}^{'}\textbf{R}^{2p-1}\pi_{*}^{'}\mathbb{Q}_{\ell}[b] \arrow[bend left=7.5, "\sim"]{r}
  & j_{*}f_{U*}\textbf{R}^{2p-1}\pi_{*}^{'}\mathbb{Q}_{\ell}[b] \arrow[bend left=7.5, "\sim"]{l}\\
  f_{*}j_{!*}^{'}\textbf{R}^{2p-1}\pi_{*}^{'}\mathbb{Q}_{\ell}[b] \arrow[dashrightarrow, bend left=7.5]{r} \arrow[u]
  & j_{!*}f_{U*}\textbf{R}^{2p-1}\pi_{*}^{'}\mathbb{Q}_{\ell}[b] \arrow[dashrightarrow, bend left=7.5]{l} \arrow[u]\\
  f_{*}j_{!}^{'}\textbf{R}^{2p-1}\pi_{*}^{'}\mathbb{Q}_{\ell}[b] \arrow[bend left=7.5, "\sim"]{r} \arrow[u]
  & j_{!}f_{U*}\textbf{R}^{2p-1}\pi_{*}^{'}\mathbb{Q}_{\ell}[b] \arrow[bend left=7.5, "\sim"]{l}  \arrow[u] 
\end{tikzcd}
\end{center}
which if we pullback to $U$ we get
\begin{center}
\begin{tikzcd}
  A \arrow[bend left=12.5]{r}
  & A \arrow[bend left=12.5]{l}\\
  A \arrow[dashrightarrow, bend left=12.5]{r} \arrow[u, "\sim"] \arrow[ur, "h"] \arrow[dr, "h'"']
  & A \arrow[dashrightarrow, bend left=12.5]{l} \arrow[u, "\text{id}"']\\
  A \arrow[bend left=12.5]{r} \arrow[u, "\sim"]
  & A \arrow[bend left=12.5]{l}  \arrow[u, "\text{id}"'] 
\end{tikzcd}
\end{center}
and therefore $h=h'$, which concludes the proof that the diagram \eqref{eq:DualDiag} commutes.\\

If we now dualize \eqref{eq:DualDiag}, we get
\begin{center}
\begin{tikzcd}
  f_{*}j_{*}^{'}\textbf{R}^{2q-1}\pi_{*}^{'}\mathbb{Q}_{\ell}[b] \arrow[bend left=7.5, "\sim"]{r}
  & j_{*}f_{U*}\textbf{R}^{2q-1}\pi_{*}^{'}\mathbb{Q}_{\ell}[b] \arrow[bend left=7.5, "\sim"]{l}\\
  f_{*}j_{!*}^{'}\textbf{R}^{2q-1}\pi_{*}^{'}\mathbb{Q}_{\ell}[b] \arrow[dashrightarrow, bend left=7.5]{r} \arrow[u]
  & j_{!*}f_{U*}\textbf{R}^{2q-1}\pi_{*}^{'}\mathbb{Q}_{\ell}[b] \arrow[dashrightarrow, bend left=7.5]{l}  \arrow[u] 
\end{tikzcd}
\end{center}
which tells us that the dual of the pushforward map is the pullback map and vice versa.
\subsection{Proving the Projection Formula}
For any map $L\rightarrow M$ in the derived category we have a commutative diagram
\begin{center}
\begin{tikzcd}
  L \arrow[r] \arrow[d]
  & D(D(L)) \arrow[d]\\
  M \arrow[r]
  & D(D(M))
\end{tikzcd}
\end{center}
where the horizontal maps are the duality morphisms. Applying this to the pushforward map just constructed we obtain
\begin{center}
\begin{tikzcd}
  f_{*}j_{!*}^{'}\textbf{R}^{2p-1}\pi_{*}^{'}\mathbb{Q}_{\ell}[b] \arrow[r] \arrow[d]
  & D(f_{*}j_{!*}^{'}\textbf{R}^{2q-1}\pi_{*}^{'}\mathbb{Q}_{\ell}[b]) \arrow[d]\\
  j_{!*}\textbf{R}^{2p-1}\pi_{*}\mathbb{Q}_{\ell}[b] \arrow[r]
  & D(j_{!*}\textbf{R}^{2q-1}\pi_{*}\mathbb{Q}_{\ell}[b])
\end{tikzcd}
\end{center}
where by the above remarks, $j_{!*}\textbf{R}^{2q-1}\pi_{*}\mathbb{Q}_{\ell}[b]\rightarrow f_{*}j_{!*}^{'}\textbf{R}^{2q-1}\pi_{*}^{'}\mathbb{Q}_{\ell}[b]$ is the pullback map.
\\

Let us apply $H^{1-b}_{\text{\'et}}(B,-)$ to this diagram, we get:
\begin{center}
\begin{tikzcd}
  H^{1-b}_{\text{\'et}}(B',j_{!*}^{'}\textbf{R}^{2p-1}\pi_{*}^{'}\mathbb{Q}_{\ell}[b]) \arrow[r, "\widehat{\cdot}"] \arrow[d, "g_{*}"']
  & \text{Hom}(j_{!*}^{'}\textbf{R}^{2q-1}\pi_{*}^{'}\mathbb{Q}_{\ell}[b],\mathbb{Q}_{\ell}[1+b]) \arrow[d, "g^{*}"]\\
  H^{1-b}_{\text{\'et}}(B,j_{!*}\textbf{R}^{2p-1}\pi_{*}\mathbb{Q}_{\ell}[b]) \arrow[r, "\widehat{\cdot}"']
  & \text{Hom}(j_{!*}\textbf{R}^{2q-1}\pi_{*}\mathbb{Q}_{\ell}[b],\mathbb{Q}_{\ell}[1+b])
\end{tikzcd}
\end{center}
So if we take $\alpha\in H^{1-b}_{\text{\'et}}(B',j_{!*}^{'}\textbf{R}^{2p-1}\pi_{*}^{'}\mathbb{Q}_{\ell}[b])$, we have that
\begin{align}
\label{eq:PushPull}
\widehat{g_{*}\alpha}=g^{*}\widehat{\alpha}
\end{align}
where by definition $g^{*}\widehat{\alpha}$ is the composition
\begin{align*}
j_{!*}\textbf{R}^{2q-1}\pi_{*}\mathbb{Q}_{\ell}[b]\xrightarrow{g^{*}} f_{*}j_{!*}^{'}\textbf{R}^{2q-1}\pi_{*}^{'}\mathbb{Q}_{\ell}[b]\xrightarrow{f_{*}(\widehat{\alpha})} f_{*}\mathbb{Q}_{\ell}[1+b]\xrightarrow{f_{*}} \mathbb{Q}_{\ell}[1+b]
\end{align*}
Let us apply $H^{1-b}_{\text{\'et}}(B,-)$ again, but now to the equality \eqref{eq:PushPull}, we compute that
\begin{align*}
H^{1-b}_{\text{\'et}}(B,-)(\widehat{g_{*}\alpha}):H^{1-b}_{\text{\'et}}(B,j_{!*}\textbf{R}^{2q-1}\pi_{*}\mathbb{Q}_{\ell}[b])\rightarrow H^{2}_{\text{\'et}}(B,\mathbb{Q}_{\ell})
\end{align*}
sends $\beta\mapsto h_{X}(g_{*}\alpha,\beta)$.\\

On the other hand, $H^{1-b}_{\text{\'et}}(B,-)(g^{*}\widehat{\alpha})$ is the composite map
\begin{center}
\begin{tikzcd}
  H^{1-b}_{\text{\'et}}(B,j_{!*}\textbf{R}^{2q-1}\pi_{*}\mathbb{Q}_{\ell}[b]) \arrow[r] \arrow[d, "g^{*}"']
  & H^{2}_{\text{\'et}}(B,\mathbb{Q}_{\ell})\\
  H^{1-b}_{\text{\'et}}(B',j_{!*}^{'}\textbf{R}^{2q-1}\pi_{*}^{'}\mathbb{Q}_{\ell}[b]) \arrow[r]
  & H^{2}_{\text{\'et}}(B',\mathbb{Q}_{\ell}) \arrow[u, "f_{*}"']
\end{tikzcd}
\end{center}
which sends $\beta\mapsto g^{*}\beta\mapsto h_{X'}(\alpha,g^{*}\beta)\mapsto f_{*}h_{X'}(\alpha,g^{*}\beta)$.\\

Putting everything together then, we have that
\begin{align*}
h_{X}(g_{*}\alpha,\beta)=f_{*}h_{X'}(\alpha,g^{*}\beta)
\end{align*}
completing the proof of \cref{prop:ProjForm}.


\begin{thebibliography}{9}

\bibitem{Bei}
A. Beilinson,
\textit{Height pairing between algebraic cycles}, K-Theory, Arithmetic and Geometry, pp. 1-26, (1987).

\bibitem{Bei-Del}
A. Beilinson, J. Bernstein, P. Deligne,
\textit{Faisceaux pervers}, Ast\'erisque, Vol. 100, pp. 5-171, (1982).

\bibitem{Cat}
M.A.A de Cataldo, L. Migliorini,
\textit{The decomposition theorem, perverse sheaves and the topology of algebraic maps}, Bulletin of the American Mathematical Society (2009).

\bibitem{dJo}
A.J. de Jong,
\textit{Smoothness, semi-stability and alterations}, Publications Math\'ematiques de l'IH\'ES, Vol. 83, pp. 51-93, (1996).

\bibitem{Ivo-Mor}
F. Ivorra, S. Morel,
\textit{The four operations on perverse motives}, Journal of the European Mathematical Society, Vol. 126, No. 11, pp. 4191-4272, (2024).

\bibitem{Kah}
B. Kahn,
\textit{Refined Height Pairing}, Algebra and Number Theory, Vol. 18, No. 6, pp. 1039-1079, (2024).

\bibitem{Kie}
R. Kiehl, R. Weissauer,
\textit{Weil Conjectures, Perverse Sheaves and l'adic Fourier Transform}, Germany: Springer, (2001).

\bibitem{Ros-Sam}
D. R\"ossler, T. Szamuely,
\textit{A generalization of Beilinson's geometric height pairing},  (2020).

\end{thebibliography}
\end{document}